\documentclass[10pt]{amsart}

\usepackage{fullpage}
\usepackage[utf8]{inputenc}
\usepackage[T1]{fontenc}

\usepackage{amsmath,amsfonts,amssymb,amsxtra,setspace,xspace,graphicx,lmodern,psfrag,epsfig,color,latexsym,mathtools}
\usepackage{subcaption}
\usepackage{dsfont}

\usepackage[colorlinks=true]{hyperref}
\hypersetup{
 urlcolor=blue,
 citecolor=red,
 pdfauthor={Harsha Hutridurga, Olga Mula and Francesco Salvarani},
 pdftitle={On the homogenization of neutron transport models},
 pdfsubject={},
 pdfkeywords={linear Boltzmann, neutron transport, homogenization, energy variable, self-shielding}
}

\graphicspath{{img/}}

\newtheorem{thm}{Theorem}
\newtheorem{lem}[thm]{Lemma}

\newtheorem{prop}[thm]{Proposition}
\newtheorem{rem}[thm]{Remark}
\newtheorem{defn}[thm]{Definition}

\newcommand{\cB}{\ensuremath{\mathcal{B}}}

\newcommand{\cK}{\ensuremath{\mathcal{K}}}
\newcommand{\cL}{\ensuremath{\mathcal{L}}}
\newcommand{\cM}{\ensuremath{\mathcal{M}}}

\newcommand{\cQ}{\ensuremath{\mathcal{Q}}}

\newcommand{\bR}{\ensuremath{\mathbb{R}}}
\newcommand{\bS}{\ensuremath{\mathbb{S}}}

\newcommand{\bV}{\ensuremath{\mathbb{V}}}

\newcommand{\bu}{\ensuremath{\textbf{u}}}

\newcommand{\bw}{\ensuremath{\textbf{w}}}

\newcommand{\rL}{\ensuremath{\mathrm{L}}}
\newcommand{\rY}{\ensuremath{\mathrm{Y}}}

\def\weak{\rightharpoonup}

\def\<{\langle}
\def\>{\rangle}
\def\({\left(}
\def\){\right)}
\newcommand{\be}{\begin{equation}}
\newcommand{\ee}{\end{equation}}

\newcommand{\dE}{\ensuremath{\,\rm dE}}
\newcommand{\domega}{\ensuremath{\,\rm d}\omega}

\newcommand{\per}{\ensuremath{\rm per}}
\newcommand{\Lper}{\ensuremath{\rL^2_{\per}}}

\newcommand{\R}{{\mathbb R}}
\renewcommand{\S}{{\mathbb S}}

\newcommand{\eps}{{\varepsilon}}

\newcommand{\Emin}{\ensuremath{E_{\min}}}
\newcommand{\Emax}{\ensuremath{E_{\max}}}
\newcommand{\Eint}{\ensuremath{(\Emin, \Emax)}}

\begin{document}

\title[Homogenization in Neutron transport]{Homogenization in the energy variable\\ for a neutron transport model}

\author{Harsha Hutridurga}
\address{H.H.: Department of Mathematics, Indian Institute of Technology Bombay, Powai, Mumbai 400076, India.}
\email{hutri@math.iitb.ac.in}

\author{Olga Mula}
\address{O.M.: Universit\'e Paris-Dauphine, PSL University, CNRS, UMR 7534, CEREMADE, 75016 Paris, France.}
\email{mula@ceremade.dauphine.fr}

\author{Francesco Salvarani}
\address{F.S.: Universit\'e Paris-Dauphine, PSL University, CNRS, UMR 7534, CEREMADE, 75016 Paris, France \& Universit\`a degli Studi di Pavia, Dipartimento di
Matematica, I-27100 Pavia, Italy} 
\email{francesco.salvarani@unipv.it}

\begin{abstract}
This article addresses the homogenization of linear Boltzmann equation when the optical parameters are highly heterogeneous in the energy variable. We employ the method of two-scale convergence to arrive at the homogenization result. In doing so, we show the induction of a memory effect in the homogenization limit. We also provide some numerical experiments to illustrate our results. In the Appendix, we treat a related case of the harmonic oscillator.
\end{abstract}

\date{\today}

\maketitle

\thispagestyle{empty}

\section{Introduction}
Neutron transport plays a key role in a number of scientific and engineering areas. It is, for example, relevant in understanding certain atmospheric processes and it also plays a major role in the field of nuclear engineering -- the safety of nuclear reactors and shielding. The evolution of neutrons in all these problems can be modelled by a linear Boltzmann equation whose coefficients, usually called optical parameters or microscopic cross sections, present rapid oscillations related to the neutron energy. In this work, we derive two-scale homogenization results related to this phenomenon which reveal the presence of a memory effect in the limit equation.

In what follows, $\Omega$ is a bounded domain of $\bR^d$ with $\mathrm{C}^1$ boundary and $\rY:=(0,1)^d$ is the unit cube in $\bR^d$. Denoting by $f = f(t,x,v)$ the \emph{population density} of neutrons which are located at $x\in\Omega$ (position) at time $t\in\bR_+$ and travelling with velocity $v\in \bV\subset\bR^d$, the neutron gas can be described by the linear Boltzmann equation (sometimes called the \emph{neutron transport equation})
\begin{equation}
\label{eq:linear-Boltzmann}
\begin{aligned}
\partial_t f + v\cdot \nabla_x f + \sigma(x,v) f - \int_{\bV} \kappa(x,v \cdot v') f(t,x,v')\, {\rm d}v' & = 0,
\end{aligned}
\end{equation}
where the non-negative functions $\sigma$ and $\kappa$ denote the \emph{total cross-section} of the background material and the \emph{scattering kernel} respectively. In the following, we will sometimes refer to the pair $(\sigma,\kappa)$ as the \emph{optical parameters}. Note that the velocity variable $v$ can be expressed via the couple $(\omega,E)$, where $\omega=v/|v|$ is the trajectory angle of the neutron and $E=m|v|^2/2$ is the kinetic energy, $m$ being the mass of the neutron. Assuming that $E$ ranges in $[E_{\min},E_{\max}]$, the same equation can equivalently be written in terms of the \emph{neutron flux} 
$$
\varphi(t,x,\omega,E) = \varphi(t,x,v) := \vert v\vert f(t,x,v),
$$
which satisfies the equation
\begin{align}
\label{eq:flux}
\sqrt{\frac{m}{2E}}\, \partial_t \varphi+ \omega \cdot \nabla_x \varphi + \sigma\left(x, \omega, E \right) \varphi
- \int_{E_{\rm min}}^{E_{\rm max}} \int_{\vert \omega'\vert=1} \kappa\left(x,\omega\cdot\omega', E,E'\right) \varphi(x,\omega',E'){\rm d}\omega'\, {\rm d}E' = 0,
\end{align}
where the optical parameters of the linear Boltzmann equation are appropriately redefined. The above evolution equation shall be supplemented by suitable initial data, i.e., $
\varphi(0,x,\omega,E) = \varphi_{\rm in}(x,\omega,E)$, and boundary data on the \emph{incoming phase-space boundary}. For simplicity, we shall consider
\emph{absorption-type} boundary data, i.e.,
\begin{align*}
\varphi(t,x,\omega,E) = 0 
\quad \forall t>0 \, \mbox{ and for }(x,\omega)\in \Gamma_-:= \left\{ (x,\omega)\in \partial\Omega\times\bS^{d-1}\ :\ {n}(x)\cdot \omega < 0\right\},
\end{align*}
where ${n}(x)$ denotes the unit exterior normal to the boundary $\partial\Omega$ at the point $x$.

The numerical computation of solutions to \eqref{eq:flux} is challenging for the following reasons:
\begin{itemize}
\item[-] the problem is high-dimensional: its solution depends on $(2d+1)$ variables and the integral term couples the angular and energy variables and gives rise to dense matrices,
\item[-] the solutions have low regularity,
\item[-] the optical parameters can present high oscillations in space (due to the spatial heterogeneity of the materials) and in energy.
\end{itemize}
The first two items have been the main focus of numerous works on the numerical analysis of equation \eqref{eq:linear-Boltzmann} and \eqref{eq:flux}. We refer to \cite{LE1984,Agoshkov1998,Kanschat2009} for general references on the topic and to a recent work \cite{DGM2018} where modern operator compression and adaptive techniques have been applied to efficiently attack the high-dimensional aspects of the problem.

The last item has motivated a considerably huge amount of literature in the theory of homogenization (see \cite{Dumas_2000} and references therein). However, to the best of our knowledge, the existing mathematical theory addresses only high oscillations in the spatial variable and no rigorous results seem to address high oscillations in the energy variable. This point has been treated thus far only in the engineering community where the problem is known as energy self-shielding or resonant absorption, which will be further commented on towards the end of this introduction. In this context, the present contribution is to bridge the gap between theory and practice by arriving at some rigorous homogenization results. This involves certain modelling assumptions on the multiscale behavior of the optical parameters which are in agreement with physical observations. Real experiments reveal strong oscillations in $\sigma$ as a function of $E$ when the neutrons interact with relevant materials like, for example, Uranium 238 (see figure \ref{img:totalXS}). A similar behaviour is also observed for the scattering kernel $\kappa$.
These facts motivate us to study the multi-scale linear Boltzmann equation
\begin{equation}
\label{eq:delta-evol-neutron-flux}
\begin{aligned}
\sqrt{\frac{m}{2E}}\, \partial_t \varphi^\varepsilon + \omega \cdot \nabla_x \varphi^\varepsilon & + \sigma^\varepsilon\left(x, \omega, E \right) \varphi^\varepsilon
- \int_{E_{\rm min}}^{E_{\rm max}} \int_{\vert \omega'\vert=1} \kappa^\varepsilon\left(x,\omega\cdot\omega', E,E'\right) \varphi^\varepsilon(x,\omega',E'){\rm d}\omega'\, {\rm d}E' = 0,
\end{aligned}
\end{equation}
where $0<\varepsilon\ll1$ is a small parameter and
\begin{align}
\label{eq:multiscaleOP}
\sigma^\varepsilon(x,\omega,E) = \sigma\left(x, \omega,E,\frac{E}{\varepsilon} \right);
\qquad
\kappa^\varepsilon(x,\omega\cdot\omega',E,E') = \kappa\left(x,\omega\cdot\omega',E,E',\frac{E'}{\varepsilon}\right),
\end{align}
where $\sigma(x,\omega,E,y)$ and $\kappa\left(x,\omega\cdot\omega',E,E',y'\right)$ are both assumed to be periodic in the $y$ and $y'$ variables respectively. The equation is complemented with zero incoming flux condition on the phase-space boundary and an initial condition $\varphi_{\text{in}}$ which we assume to be in $\rL^2(\Omega\times\S^{d-1}\times (E_{\min}, E_{\max}) )$.

In addition to the above hypotheses, we also assume that there exists $\alpha>0$ such that for all $\varepsilon >0$,
\begin{align}
\label{eq:hypOP}
\sigma^\varepsilon (x,\omega, E) - \bar\kappa^\varepsilon(x,\omega, E) \geq \alpha
\quad \text{and} \quad
\sigma^\varepsilon (x,\omega, E) - \tilde\kappa^\varepsilon(x,\omega, E) \geq \alpha
,
\end{align}
where
\begin{equation}
\label{eq:kappa_tilde}
\begin{cases}
\bar\kappa^\varepsilon(x,\omega, E) &\displaystyle \coloneqq \int_{E_{\min}}^{E_{\max}} \int_{\S^{d-1}} \kappa^\varepsilon(x,\omega\cdot\omega',E,E') \,\domega'\dE'
\\[10pt]
\tilde\kappa^\varepsilon(x,\omega, E) &\displaystyle  \coloneqq \int_{E_{\min}}^{E_{\max}} \int_{\S^{d-1}} \kappa^\varepsilon(x,\omega\cdot\omega',E',E) \,\domega'\dE'.
\end{cases}
\end{equation}
From a physical point of view, these assumptions mean that we place ourselves in the so-called subcritical regime where absorption phenomena dominate scattering. 

\begin{figure}[h]
\includegraphics[scale=0.8]{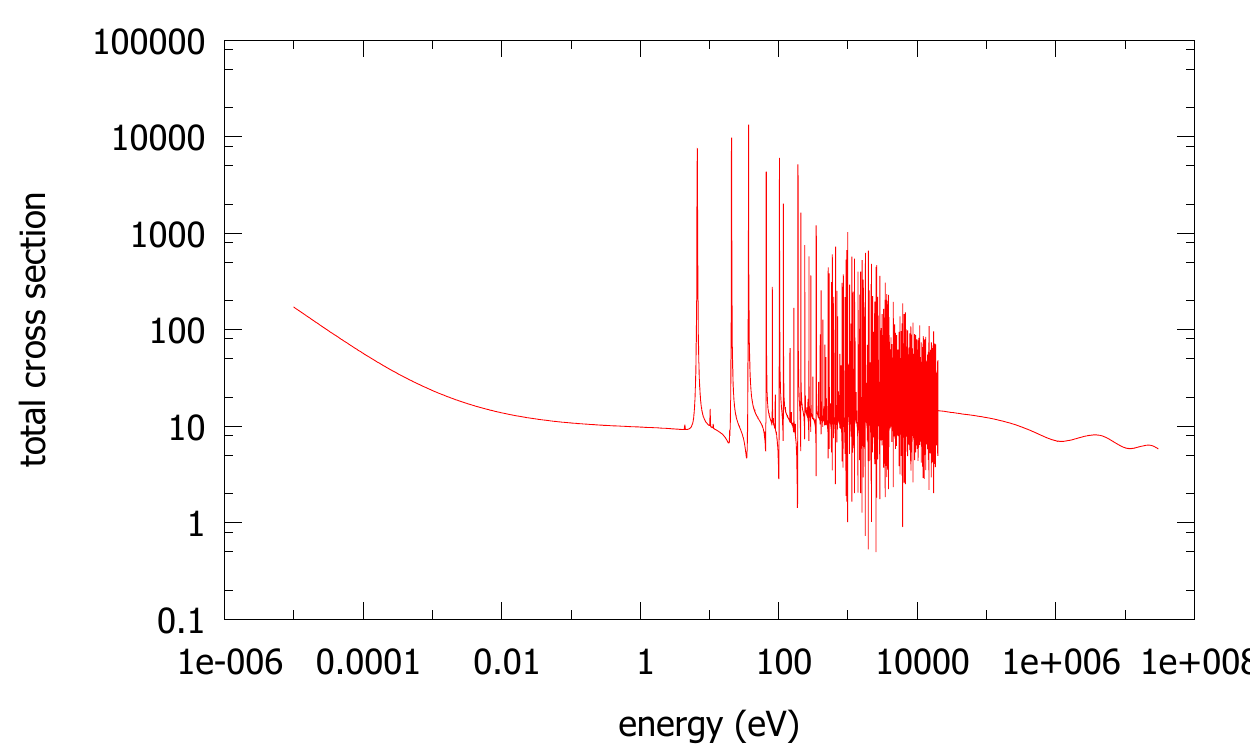}
\caption{Total cross-section $\sigma$ of Uranium 238 as a function of the energy according to the JEFF 3.1 library \cite{jeff}. Note the highly oscillatory interval for $E\in[1\text{ eV};\,10^4\text{ eV}]$.}
\label{img:totalXS}
\end{figure}

From the homogenization viewpoint, transport dominated equations such as \eqref{eq:delta-evol-neutron-flux} are particularly challenging since the structure of the partial differential equation becomes more complex after taking the homogenization limit. This is due to the memory effects induced in the limit that make the dynamics be no longer defined by a semigroup \cite{Tartar_2006, Bernard_2010, MS2013}. This, in turn, entails difficulties in the numerical solution of the homogenized equation since the memory effects dramatically increase the computational complexity in terms of the number of degrees of freedom to be used in order to retrieve a certain target accuracy.

Our main contribution is the homogenization result given in section \ref{sec:hom_linear_Boltzmann}, where we derive a homogenized equation for the neutron transport problem when the optical parameters oscillate periodically in the energy variable. The result is derived employing the theory of two-scale convergence. For technical reasons that we will explain later on in more detail, it is more convenient to work with scattering kernels $\kappa^\varepsilon$ exhibiting separation in the $E$ and $E'$ variables as follows:
$$
\kappa^\varepsilon(x, \omega\cdot\omega',E,E') : = \kappa_1(x, \omega\cdot\omega',E) \kappa_2\( x, \omega\cdot\omega',E', \frac{E'}{\varepsilon} \) 
$$
with $\kappa_2\( x, \omega\cdot\omega',E', y' \) $ being periodic in the $y'$ variable. A result without this assumptions on $\kappa^\varepsilon$ requires further investigation as it is not apparent whether one can derive closed form homogenized equations in the latter case.

The homogenized equation is integro-differential and presents a memory term. To derive the results, we base our strategy on the method of characteristics, and hence we first derive a homogenization result for the associated ordinary differential equation (Section \ref{sec:ode}). In there, we also show that this result is in agreement with previous works on memory effects by Tartar \cite{Tartar1989}, \cite[chapter 35]{Tartar_2006}. An interesting result in its own right is that our technique gives an explicit expression of the memory kernel that, in the situation studied by Tartar, is equal to the implicit expression given in \cite{Tartar1989}, \cite[chapter 35]{Tartar_2006}. Finally we consider in section \ref{sec:numerics} some numerical experiments on a simpler model, mainly focused on the convergence rate towards the homogenized solution as the period of the oscillating terms tends to zero (i.e., rapid oscillations). In the Appendix, we treat the related case of a harmonic oscillator.

We conclude this introduction by comparing our approach to some standard methods from the nuclear engineering community for treating self-shielding phenomena. To the best of the authors' knowledge, the most widespread technique is a two-stage method originially proposed in \cite{LJ1974} by M.~Livolant and F.~Jeanpierre (we refer to \cite[Chapters 8 and 15]{Reuss2008} for an introductory overview). It consists in finding first the averaged optical parameters which are then plugged into a multigroup version of equation \eqref{eq:linear-Boltzmann} to compute the behavior of the flux on large and geometrically complex domains such as nuclear reactors. The pre-computation of the averaged parameters is done on a cell problem involving a much simpler spatial geometry and simplified physics. It is nevertheless carefully designed with elaborate physical considerations in a way to keep as much consistency as possible with respect to the original problem. Note that, while this approach implicitly assumes that the homogenized equation is of the same nature as the original Boltzmann problem \eqref{eq:linear-Boltzmann}, our starting point is fundamentally different in the sense that we do not postulate any final form of the limit equation. Our goal is precisely to discover its form from the only assumption that the optical parameters oscillate in energy. As a result, our methodology and conclusions are different from the ones discussed in \cite{LJ1974}
and do not involve a pre-computation on a cell problem. Another approach, also based on averaging the optical parameters, is the so-called multi-band method (see \cite{CP1980}), where, like in the previous method, the limit equation is assumed to be a Boltzmann equation. Finally, a more recent approach based on averaging arguments taken from results of homogenization of pure transport equations has recently been proposed in \cite{HAJLT2017}. The initial problem there is a time-independent Boltzmann source problem with no oscillations in the scattering kernel.


\section{Homogenization of an ordinary differential equation}
\label{sec:ode}
It turns out that the homogenization of the multi-scale linear Boltzmann equation \eqref{eq:delta-evol-neutron-flux} is closely related to the homogenization of an evolution equation studied by Luc Tartar in the 1980's. Hence, we begin this section by recalling his example and by giving an alternate proof for the phenomenon of memory effect by homogenization, as was demonstrated by him in \cite{Tartar1989}. For the unknown $u^\varepsilon$, consider the differential equation
\begin{equation}\label{eq:tartar:model}
\partial_t u^\varepsilon + \sigma\left(\frac{x}{\eps}\right) u^\eps = 0; \qquad \qquad u^\eps(0,x) = u_{\rm in}(x).
\end{equation}
Even though Tartar treats a more general setting, Theorem \ref{thm:tartar-compare} recalls the result from \cite{Tartar1989} adapted to when $\sigma$ has only periodic oscillations. In the sequel, we use the following notation for the Laplace transform (in the time variable) of a function:
\[
\widehat{f}(p) := \int_0^\infty e^{-ps} f(s)\, {\rm d}s
\qquad \mbox{ for }p>0.
\]
In the section, for any $v\in \rL^1(\rY)$,
$$
\langle v \rangle \coloneqq \int_\rY v(y) \, {\rm d}y
$$
denotes the average of $v$ in $\rY$.
\begin{thm}[Tartar, \cite{Tartar1989}]
\label{thm:tartar-compare}
Let the coefficient $\sigma(\cdot)$ in \eqref{eq:tartar:model} be a strictly positive, bounded and purely periodic coefficient of period $\mathrm Y$. Then the $\mathrm L^\infty$ weak $\ast$ limit $u_{\mathrm{hom}}(t,x)$ of the solution family $u^\eps$ satisfies the following integro-differential equation
\begin{equation}\label{eq:evol-u-hom-tartar}
\left\{
\begin{aligned}
\partial_t u_{\rm hom}(t,x) + \langle\sigma\rangle u_{\rm hom}(t,x) - \int_0^t \mathcal{M}(t-s) u_{\rm hom}(s,x) \, {\rm d}s & = 0
\\
u_{\rm hom}(0,x) &= u_{\rm in}(x)
\end{aligned}
\right.
\end{equation}
where the memory kernel $\mathcal{M}(\tau)$ is given in terms of its Laplace transform as follows
\begin{align}\label{eq:equiv:Tartar-Laplace}
\widehat{\mathcal{M}}(p) = p + \langle{\sigma}\rangle - \mathcal{B}(p) = \int_\rY \Big( p + \sigma(y) - \mathcal{B}(p) \Big)\, {\rm d}y\qquad \forall p>0,
\end{align}
with the constant $\mathcal{B}(p) $ taking the value
\begin{align*}
\mathcal{B}(p) := \left( \int_\rY \frac{{\rm d}y}{p + \sigma(y)}\right)^{-1}.
\end{align*}
\end{thm}

Note that the memory kernel is expressed implicitly by means of the Laplace transform (see formula \eqref{eq:equiv:Tartar-Laplace}). This is due to the fact that the proof given in \cite{Tartar1989} goes via taking the Laplace transform in the $t$ variable. Contrary to this, our alternative approach yields directly an explicit expression without necessitating the use of Laplace transform in the proof. This is because we place ourselves in the periodic homogenization setting, which will in turn aid us in homogenizing the neutron transport equation \eqref{eq:delta-evol-neutron-flux}.



Our starting point is the following evolution equation (we consider oscillatory sources as this will resemble more our final targeted equation \eqref{eq:delta-evol-neutron-flux}):
\begin{equation}\label{eq:exm:model}
\left\{
\begin{aligned}
\partial_t u^\eps(t,x) + \sigma^\eps(x) u^\eps(t,x) & = f^\eps (t,x) \qquad \mbox{ for }(t,x)\in(0,T)\times\Omega,
\\[0.2 cm]
u^\eps(0,x) & = u^\eps_{\rm in}\left(x\right) \quad \quad \mbox{ for }x\in\Omega,
\end{aligned}\right.
\end{equation}
where $T<\infty$ is arbitrary. The coefficients and data in \eqref{eq:exm:model} are of the form
\[
\sigma^\eps(x) := \sigma\left(x,\frac{x}{\eps}\right),\quad f^\eps(t,x) := f\left(t,x,\frac{x}{\eps}\right), \quad u^\eps_{\rm in}(x) := u_{\rm in}\left(x,\frac{x}{\eps}\right),
\]
with the locally periodic (in space) coefficient
$$
\sigma(x,y)\in \rL^\infty(\Omega;\mathrm C_{\per}(\rY))
$$
and data
$$
f(t,x,y)\in \rL^\infty((0,T)\times \Omega;\mathrm C_{\per}(\rY))  , \quad u_{\rm in}(x,y)\in\rL^2(\Omega;\mathrm C_{\per}(\rY)).
$$
Here we have used the standard notation $\mathrm C_{\rm per}(\rY)$ to denote continuous functions on $\R^d$ which are $\rY$-periodic. Furthermore we assume that there exists a positive constant $\sigma_{\rm min}$ such that
$$
\sigma(x,y) \geq \sigma_{\min}\quad \forall\, (x,y)\in \Omega\times \rY.
$$

Since our analysis relies on the notion of \emph{two-scale convergence}, before stating our homogenization result, we first recall some basic results and introduce some notation (for proofs and further discussions, readers are encouraged to refer to the works cited here). Two-scale convergence was also used by J.-S.~Jiang \cite{Jiang2006} to treat the case with $u^\eps_{\rm in}=0$ and $f^\eps(x)=f(x)$ for all $\eps>0$. In that work, the resulting memory kernel was again given implicitly, this time in terms of a certain Volterra equation.

Two-scale convergence was introduced by Gabriel Nguetseng \cite{Nguetseng1989} and further developed by Gr\'egoire Allaire \cite{Allaire_1992}. We start by recalling its definition.
\begin{defn}\label{defn:two-scale}
A family of functions $v^\eps(x)\subset \rL^2(\Omega)$ is said to \emph{two-scale converge} to a limit $v^0(x,y)\in \rL^2(\Omega\times \rY)$ if, for any smooth test function $\psi(x,y)$, $\rY$-periodic in the $y$ variable, we have
\begin{align*}
\lim_{\eps\to0} \int_\Omega v^\eps(x) \psi\left( x, \frac{x}{\eps} \right) \, {\rm d}x = \int_\Omega \int_Y v^0(x,y) \psi(x,y) \, {\rm d}x\, {\rm d}y.
\end{align*}
\end{defn}
A distinctive result in the two-scale convergence theory is the following result of compactness which says that the above notion of convergence is not void.
\begin{thm}[\cite{Nguetseng1989, Allaire_1992}]\label{thm:two-scale-compact}
Suppose a family $v^\eps(x)\subset \rL^2(\Omega)$ is uniformly bounded, i.e.,
\[
\left\Vert v^\eps \right\Vert_{\rL^2(\Omega)} \le C
\]
with constant $C$ being independent of $\eps$. Then, we can extract a sub-sequence (still denoted $v^\eps$) such that $v^\eps$ two-scale converges to some limit $v^0(x,y)\in\rL^2(\Omega\times Y)$.
\end{thm}
We additionally have the following useful property of the two-scale limit. 
\begin{prop}[\cite{Nguetseng1989, Allaire_1992}]
\label{prop:weaklimit2scale}
Let $v^\eps$ be a sequence of functions in $\rL^2(\Omega)$ which two-scale converges to a limit $v^0\in\rL^2(\Omega\times\rY)$. Then $v^\eps(x)$ converges to $\langle{v}\rangle(x)=\int_\rY v^0(x,y)\, {\rm d}y$ weakly in $\rL^2(\Omega)$, i.e.,
\[
\lim_{\eps\to0} \int_\Omega v^\eps(x) \varphi(x)\, {\rm d}x = \int_\Omega \varphi(x) \int_Y v^0(x,y) \, {\rm d}y\, {\rm d}x
\qquad \mbox{ for all }\varphi \in \mathrm L^2(\Omega).
\]
\end{prop}
Note that the notion of two-scale convergence is a weak-type convergence as it is given in terms of test functions (see Definition \ref{defn:two-scale}). 

For any given $g\in \mathrm L^\infty(Y)$, the linear operator
\[
\mathcal{L}_g\, v := gv - \< g v \> \quad \forall v\in\rL^2_{\per}(\rY)
\]
is bounded in $\Lper(\rY)$ as
\begin{align*}
\left\Vert \mathcal{L}_g h \right\Vert^2_{\rL^2_{\per}(\rY)}
= \int_\rY \left\vert g(y) h(y) - \langle gh \rangle \right\vert^2\, {\rm d}y
= \int_\rY \left\vert g(y) h(y) \right\vert^2\, {\rm d}y - \langle g h \rangle^2
\end{align*}
and we have by Cauchy-Schwarz
\begin{align*}
\left\vert \langle g h \rangle \right\vert = \left\vert \int_\rY g(y) h(y)\, {\rm d}y \right\vert \le \left( \int_\rY \left\vert g(y) h(y) \right\vert^2\, {\rm d}y\right)^\frac12.
\end{align*}
As a consequence, $\cL_g: \Lper(\rY) \mapsto \Lper(\rY)$ is the infinitesimal generator of a uniformly continuous semigroup given by
$$
e^{t\cL_g} = \sum_{n=0}^\infty \frac{t^n}{n!} \cL_g^n.
$$

We next present the main homogenization result for the evolution \eqref{eq:exm:model}.\begin{thm}
\label{thm:exm:ode}
Let $u^\eps(t,x)$ be the solution to the evolution problem \eqref{eq:exm:model}. Then
\[
u^\eps \weak u_{\rm hom} \quad \mbox{ weakly in }\rL^2((0,T)\times\Omega)
\]
and $u_{\rm hom}(t,x)$ solves the following integro-differential equation
\begin{equation}\label{eq:evol-u-hom}
\left\{
\begin{aligned}
\partial_t u_{\rm hom}(t,x) + \langle\sigma\rangle(x) u_{\rm hom}(t,x) - \int_0^t \mathcal{K}(t-s,x) u_{\rm hom}(s,x) \, {\rm d}s & = \mathcal{S}(t,x)
\\
u_{\rm hom}(0,x) &= \langle u_{\rm in} \rangle(x)
\end{aligned}
\right.
\end{equation}
where the memory kernel is given by
\be
\label{eq:memoryKernel}
\mathcal{K}(\tau,x) \coloneqq \int_\rY \sigma(x,y) e^{-\tau\mathcal{L}_\sigma}\mathcal{L}_1 \sigma(x,y)\, {\rm d}y
\ee
and the source term is given by
\be
\label{eq:homogRhs}
\mathcal{S}(t,x)
\coloneqq
\langle f \rangle(t,x)
- \int_0^t \int_\rY \sigma(x,y) e^{-(t-s)\mathcal{L}_\sigma} \mathcal{L}_1 f (s,x,y) \, {\rm d}y \, {\rm d}s 
- \int_\rY \sigma(x,y) e^{-t\mathcal{L}_\sigma} \mathcal{L}_1 u_{\rm in}(x,y)\, {\rm d}y.
\ee
\end{thm}

\begin{rem}
The expression \eqref{eq:memoryKernel} should be understood as the action of the semigroup $e^{-t\mathcal{L}_\sigma}$ on the quantity $\sigma(y) - \langle{\sigma}\rangle$ followed by the multiplication of the periodic function $\sigma$. The thus obtained result is then integrated over the periodic variable $y$ in the unit cell $\mathrm Y$. Similar interpretation can be given for the terms involved in the expression \eqref{eq:homogRhs} for the source $\mathcal{S}(t,x)$.
\end{rem}

\begin{proof}[Proof of Theorem \ref{thm:exm:ode}]
To employ compactness results of the two-scale convergence theory, we first derive uniform (w.r.t $\eps$) estimates on the solution $u^\eps(t,x)$.
Owing to the simplicity of the evolution equation \eqref{eq:exm:model}, we can explicitly solve it and get
\begin{align*}
u^\eps(t,x) = u_{\rm in}\left(x,\frac{x}{\eps}\right) e^{-\sigma^\eps(x)t} + \int_0^t e^{-\sigma^\eps(x)(t-s)} f\left(s,x,\frac{x}{\eps}\right)\, {\rm d}s.
\end{align*}
The regularity properties of the initial condition $u_{\rm in}$ and of the source term $f$, together with the fact that $\sigma^\eps \geq 0$, imply that
$$
\Vert u^\eps \Vert_{\rL^\infty((0,T); \rL^2( \Omega))} \leq C <\infty,\quad
$$
where the constant is $C$ independent of the parameter $\eps$. Therefore, by Theorem \ref{thm:two-scale-compact}, there exists a subsequence $u^\eps$ which two-scale converges to a function $u^0 \in \rL^2((0,T)\times\Omega\times \rY)$.\\
Next, 
Passing to the limit as $\eps\to0$ in the sense of two-scale in the above expression, we obtain
\begin{align*}
u^0(t,x,y) = u_{\rm in}(x,y) e^{-\sigma(x,y)t} + \int_0^t e^{-\sigma(x,y)(t-s)} f\left(s,x,y\right)\, {\rm d}s
\end{align*}
which essentially says that the limit $u^0$ solves the following two-scale evolution:
\begin{equation}\label{eq:exm:2scale-model}
\left\{
\begin{aligned}
\partial_t u^0(t,x,y) + \sigma(x,y) u^0(t,x,y) &= f(t,x,y) \qquad \mbox{ for }(t,x,y)\in(0,T)\times\Omega\times\mathrm Y,
\\[0.2 cm]
u^0(0,x,y) &= u_{\text{in}}\left(x,y\right)\qquad \mbox{ for }(x,y)\in\Omega\times\mathrm Y.
\end{aligned}
\right.
\end{equation}
By Proposition \ref{prop:weaklimit2scale}, the sequence $u^\eps$ converges weakly in $\rL^2((0,T)\times\Omega)$  to
\[
u_{\rm hom}(t,x) \coloneqq \langle u^0\rangle (t,x)
\]
and we can then decompose the two-scale limit into a homogeneous part and a remainder which is of zero mean over the periodic cell, i.e.,
\begin{align}\label{eq:exm:2scale-decomposition}
u^0(t,x,y) = u_{\rm hom}(t,x) + r(t,x,y)
\qquad 
\mbox{ where }
\quad 
\langle r \rangle = 0.
\end{align}
Substituting \eqref{eq:exm:2scale-decomposition} into the evolution equation for the two-scale limit \eqref{eq:exm:2scale-model}, we obtain
\begin{equation}
\label{eq:exm:substiute}
\partial_t u_{\rm hom} + \sigma(x,y) u_{\rm hom}
+
\partial_t r + \sigma(x,y) r
=
f(t,x,y).
\end{equation}
Integrating the above equation over the periodicity cell $\rY$ yields
\begin{equation}
\label{eq:exm:substiute-average}
\partial_t u_{\rm hom} + \langle{\sigma}\rangle(x) u_{\rm hom}
=
\langle{f}\rangle(t,x) - \< \sigma(x,\cdot) r(t,x,\cdot)\>
\end{equation}
as the reminder $r$ is of zero average in the $y$ variable. Using \eqref{eq:exm:substiute-average} in \eqref{eq:exm:substiute} we obtain an equation for the remainder term:
\begin{align}\label{eq:r}
\partial_t r + \sigma(x,y) r - \int_\rY \sigma(x,y) r(t,x,y)\, {\rm d}y = \Big( \langle{\sigma}\rangle(x) - \sigma(x,y) \Big) u_{\rm hom} + f(t,x,y) - \langle{f} \rangle (t,x).
\end{align}
As a result, equations \eqref{eq:exm:substiute-average} and \eqref{eq:r} yield the following coupled system for $u_{\rm hom}(t,x)$ and $r(t,x,y)$:
\begin{equation}
\label{eq:exm:coupled-system}
\left\{
\begin{aligned}
\partial_t u_{\rm hom} + \langle{\sigma}\rangle (x) u_{\rm hom}
& =
\langle{f}\rangle(t,x) - \< \sigma(x,\cdot) r(t,x,\cdot)\>
\\
\partial_t r + \mathcal{L}_\sigma r & = - u_{\rm hom} \mathcal{L}_1\sigma + \mathcal{L}_1 f
\\
u_{\rm hom}(0,x) &= \langle u_{\rm in}(x) \rangle
\\
r(0,x,y) &= \mathcal{L}_1 u_{\rm in}.
\end{aligned}\right.
\end{equation}
In the above couple system, we can solve for the remainder term $r(t,x,y)$ in terms of the homogenized limit $u_{\rm hom}$ as follows:
\begin{align}
r(t,x,y) = e^{-t\mathcal{L}_\sigma} \mathcal{L}_1 u_{\rm in}(x,y) 
+ \int_0^t e^{-(t-s)\mathcal{L}_\sigma} \mathcal{L}_1 f(s,x,y)\, {\rm d}s
- \int_0^t e^{-(t-s)\mathcal{L}_\sigma} \mathcal{L}_1 \sigma(x,y) u_{\rm hom}(s,x) \, {\rm d}s.
\end{align}
Substituting this expression for the remainder in the evolution for $u_{\rm hom}(t,x)$ in \eqref{eq:exm:coupled-system} yields
\begin{align*}
\partial_t u_{\rm hom} + \langle{\sigma}\rangle(x) u_{\rm hom}
&=
\langle{f}\rangle(t,x) - \int_\rY \sigma(x,y) v(t,x,y)\, {\rm d}y - \int_\rY \sigma(x,y) w(t,x,y)\, {\rm d}y \\
&= \langle{f}\rangle(t,x) 
+ \int_0^t \int_\rY \sigma(x,y) e^{-(t-s)\mathcal{L}_\sigma} \mathcal{L}_1 \sigma(x,y) u_{\rm hom}(s,x) \, {\rm d}y \, {\rm d}s 
\\
&\quad- \int_0^t \int_\rY \sigma(x,y) e^{-(t-s)\mathcal{L}_\sigma} \mathcal{L}_1 f (s,x,y) \, {\rm d}y \, {\rm d}s 
- \int_\rY \sigma(x,y) e^{-t\mathcal{L}_\sigma} \mathcal{L}_1 u_{\rm in}(x,y)\, {\rm d}y,
\end{align*}
from which the evolution equation \eqref{eq:evol-u-hom} for $u_{\rm hom}(t,x)$ easily follows.
\end{proof}

\begin{rem}
The result in Theorem \ref{thm:exm:ode} can easily be extended to a system of ordinary differential equations of the form 
\begin{equation}\label{eq:exm:modelSystem}
\left\{
\begin{aligned}
\partial_t \bu^\eps(t,x) + {\bf \Sigma}^\eps(x) \bu^\eps(t,x) & = {\bf f}^\eps (t,x)
\\[0.2 cm]
\bu^\eps(0,x) & = \bu_{\rm in}\left(x,\frac{x}{\eps}\right)
\end{aligned}\right.
\end{equation}
where the unknown $\bu^\eps \in \bR^n$. In this case,
$$
{\bf\Sigma}^\eps(x) := {\bf\Sigma}\left(x,\frac{x}{\eps}\right) = \( \Sigma_{i,j} \left(x,\frac{x}{\eps}\right)  \)_{1\leq i,j\leq n}, \quad \Sigma_{i,j} \in \mathrm L^\infty(\Omega;\mathrm C_{\per}(\rY))
$$
and we assume that ${\bf\Sigma}^\eps$ is diagonalizable in the sense that there exists ${\bf P} \in \bR^{n\times n}$ invertible and ${\bf D}^\eps\in \bR^{n\times n}$ diagonal such that ${\bf\Sigma}^\eps = {\bf P D^\eps P}^{-1}$. The right hand side is now
$$
{\bf f}^\eps(t,x) := {\bf f}\left(t,x,\frac{x}{\eps}\right) = \( f_i  \left(t,x,\frac{x}{\eps}\right) \)_{1\leq i\leq n}, \quad f^\eps_{i} \in \mathrm L^\infty((0,T)\times \Omega;\mathrm C_{\per}(\rY)).
$$
Since $\bw^\eps = {\bf P}^{-1}\bu^\eps = (w^\eps_i)_{1\leq i \leq n}$ satisfy for $1\leq i \leq n$,
\begin{equation}
\label{eq:exm:modelsystem}
\left\{
\begin{aligned}
\partial_t w_i^\eps(t,x) + D_{i,i}\(x\frac{x}{\eps}\) w_i^\eps(t,x) & = \({\bf P}^{-1}{\bf f}\)_i \(t,x,\frac{x}{\eps}\)
\\[0.2 cm]
w_i^\eps(0,x) & =  \({\bf P}^{-1}{\bf u}_{\rm in}\)_i \left(x,\frac{x}{\eps}\right),
\end{aligned}\right.
\end{equation}
we can now apply Theorem $\ref{thm:exm:ode}$ to problem \eqref{eq:exm:modelsystem} and derive the homogenized equation and corresponding solution $w_{i, \hom}$. We derive the final result by undoing the change of variables, namely by computing $\bu_{\hom} = {\bf P} \bw_{\hom}$.

Note that our extension required that $\bf \Sigma$ be diagonalizable in $\bR$. At the end of the paper (see Appendix \ref{sec:complex}), we present an extension of our ODE setting \eqref{eq:exm:model} to the case where $\sigma$ and $f$ lie in the complex plane which allows us to work with any matrix $\bf\Sigma$.
\end{rem}

We next show that expression \eqref{eq:memoryKernel} of our memory kernel is consistent with the one given by Tartar in \eqref{eq:equiv:Tartar-Laplace}. For this, note that when $\sigma(x,y)=\sigma(y)$, i.e., when the coefficient is purely periodic, the memory kernel $\mathcal{K}$ takes the form
\begin{align}\label{eq:equiv:Our-expression}
\mathcal{K}(\tau) := \int_\rY \sigma(y) e^{-\tau\mathcal{L}_\sigma} \left( \sigma - \langle{\sigma}\rangle \right)(y)\, {\rm d}y.
\end{align}
The following result shows that, in this setting, we have $\widehat{\cK} = \widehat{\cM}$, which means that our expression coincides with that of Tartar.

\begin{thm}
For any $p>0$,
$$
\widehat{\mathcal{K}}(p) = \widehat{\mathcal{M}}(p).
$$
\end{thm}

\begin{proof}
The Laplace transform of our expression \eqref{eq:equiv:Our-expression} for the memory kernel is
\begin{align*}
\widehat{\mathcal{K}}(p) = \int_0^\infty e^{-pt} \mathcal{K}(t) \, {\rm d}t = \int_0^\infty \int_\rY \sigma(y) e^{-pt} e^{-t\mathcal{L}_\sigma} \left( \sigma - \langle{\sigma}\rangle \right)(y)\, {\rm d}y \, {\rm d}t
\end{align*}
and since the Laplace transform of a semigroup yields the corresponding resolvent, we have
\begin{align}\label{eq:equiv:Our-expression-Laplace}
\widehat{\mathcal{K}}(p) = \int_\rY \sigma(y) \left[ p + \mathcal{L}_\sigma \right]^{-1} \left( \sigma - \langle{\sigma}\rangle \right)(y)\, {\rm d}y.
\end{align}
%
We now prove that
\begin{equation}\label{eq:laplaceEq}
\left[ p + \mathcal{L}_\sigma \right]^{-1} \left( \sigma - \langle{\sigma}\rangle \right)(y)
=
\frac{\sigma(y) - \langle{\sigma}\rangle}{p + \sigma} - \frac{\mathcal{B}(p)}{p + \sigma} \int_\rY \frac{\sigma(y) - \langle{\sigma}\rangle}{p + \sigma(y)}\, {\rm d}y. 
\end{equation}
For this, we consider the equation
\begin{align*}
\left[ p + \mathcal{L}_\sigma \right] g(y) = f(y),\quad y\in \rY
\end{align*}
for a given $p>0$ and a given measurable function $f$ of zero mean.
%
%
Averaging the equation in the $y$ variable yields that it is necessary to have the solution $g(y)$ to be of zero average as well. Hence, for zero average functions, we observe that
\begin{align*}
\left[ p + \mathcal{L}_\sigma \right] g(y) = p + \sigma(y) g(y) - \int_\rY \sigma(y) g(y)\, {\rm d}y = p + \sigma(y) g(y) - \int_\rY \left( \sigma(y) + p \right) g(y)\, {\rm d}y = \mathcal{L}_{p+\sigma} g(y).
\end{align*}
A simple inspection reveals that a general solution to
\[
\mathcal{L}_{p+\sigma} g(y) = f(y)
\]
is given by
\begin{align*}
g(y) = \frac{f(y)}{p+\sigma(y)} + \frac{C}{p+\sigma(y)}
\end{align*}
where the constant $C$ needs to be chosen such that $g(y)$ is of zero average. Hence
$$
C = - \cB(p)  \int_\rY \frac{f(y)}{p+\sigma(y)} \, {\rm d}y
\qquad \mbox{ with }\quad
\mathcal{B}(p) := \left( \int_\rY \frac{{\rm d}y}{p + \sigma(y)}\right)^{-1}
$$
and we have found the explicit expression for the solution $g(y)$.

As a result, by taking $f(y) = \sigma - \langle{\sigma}\rangle$, we obtain the solution to be 
\begin{align*}
\left[ p + \mathcal{L}_\sigma \right]^{-1} \left( \sigma - \langle{\sigma}\rangle \right)(y) = g(y) = \frac{\sigma(y) - \langle{\sigma}\rangle}{p + \sigma} - \frac{\mathcal{B}(p)}{p + \sigma} \int_\rY \frac{\sigma(y) - \langle{\sigma}\rangle}{p + \sigma(y)}\, {\rm d}y
\end{align*}
as anticipated in \eqref{eq:laplaceEq}.
Using the above observation, the expression \eqref{eq:equiv:Our-expression-Laplace} yields
\begin{equation}\label{eq:equiv:laplace-calcul-1}
\begin{aligned}
\widehat{\mathcal{K}}(p) & = \int_\rY \sigma(y) \left( \frac{\sigma(y) - \langle{\sigma}\rangle}{p + \sigma} - \frac{\mathcal{B}(p)}{p + \sigma} \int_\rY \frac{\sigma(y) - \langle{\sigma}\rangle}{p + \sigma(y)}\, {\rm d}y \right)\, {\rm d}y
\\
& = \int_\rY \left( \sigma(y) + p \right) \left( \frac{\sigma(y) - \langle{\sigma}\rangle}{p + \sigma} - \frac{\mathcal{B}(p)}{p + \sigma} \int_\rY \frac{\sigma(y) - \langle{\sigma}\rangle}{p + \sigma(y)}\, {\rm d}y \right)\, {\rm d}y
\end{aligned}
\end{equation}
as the function $g(y)$ is of zero mean. Observe that the following quotient is of zero mean
\[
\frac{p+\sigma(y) - \mathcal{B}(p)}{p + \sigma(y)},
\]
which implies the following identity
\begin{align}
\int_\rY \frac{\sigma(y)}{p + \sigma(y)}\, {\rm d}y = 1 - \frac{p}{\mathcal{B}(p)}.
\end{align}
Using it in \eqref{eq:equiv:laplace-calcul-1} yields
\begin{align*}
\widehat{\mathcal{K}}(p) = \langle{\sigma}\rangle - \mathcal{B}(p) \int_\rY \frac{\sigma(y)}{p + \sigma(y)}\, {\rm d}y = p + \langle{\sigma}\rangle - \mathcal{B}(p) = \widehat{\mathcal{M}}(p),
\end{align*}
thus proving the equivalence.
\end{proof}

\section{Homogenization in energy for the linear Boltzmann equation}
\label{sec:hom_linear_Boltzmann}

We consider the homogenization of the multi-scale linear Boltzmann model \eqref{eq:delta-evol-neutron-flux}, where the optical parameters $\sigma^\varepsilon$ and $\kappa^\varepsilon$ present oscillations in the energy variable that are modeled in accordance to experimental observations (see the Introduction). The first step is the proof of some a priori bounds, which will then be used to prove our main homogenization result. The proofs for the a priori bounds involve classical techniques and similar results can be found in \cite[Chapter 21]{DLvol6}.\\
In this section, $\langle \cdot \rangle$ denotes integral over the interval $(0,1)$, i.e. averaging over the periodic cell in the energy variable:
\[
\langle v \rangle := \int_0^1 v(y)\, {\rm d}y \qquad \quad \mbox{ for all }v\in\mathrm L^1(0,1).
\]

\subsection{A priori bounds of the solution}
In the following, we will sometimes use the shorthand notation
$$
\bV = \S^{d-1}\times \Eint
$$
We first prove a positivity property of the Boltzmann operator
\begin{align*}
\cQ^\varepsilon f := \sigma^\varepsilon \, f - \int_{\bV} \kappa^\varepsilon(x, \omega\cdot \omega',E,E') f(x,\omega',E')\, {\rm d}\omega'\, {\rm d}E',\quad \forall f \in \rL^2(\Omega\times\bV).
\end{align*}

\begin{prop}
\label{prop:positivityQ}
If $(\sigma^\varepsilon, \kappa^\varepsilon)$ satisfy assumptions \eqref{eq:multiscaleOP} and \eqref{eq:hypOP}, then for all $\varepsilon>0$ and all $f\in \rL^2(\Omega\times\bV)$,
\begin{equation}
\left( \cQ^\varepsilon f, f \right)_{\rL^2(\Omega\times\bV)} \geq \alpha \Vert f \Vert^2_{ \rL^2(\Omega\times\bV)}.
\end{equation}
\end{prop}

\begin{proof}
It follows from the Cauchy-Schwarz inequality and the definition of $\bar\kappa^\varepsilon$ and $\tilde\kappa^\varepsilon$ in \eqref{eq:kappa_tilde} that
\begin{align*}
&\int_{\bV}\int_{\bV} f(x,\omega,E) f(x,\omega',E') \kappa^\varepsilon(x,\omega\cdot\omega',E, E') \domega\dE \domega'\dE' \\
&\hspace{2.0cm}
\leq
\left( \int_{\bV}\int_{\bV}  \left\vert f(x,\omega,E) \right\vert^2 \kappa^\varepsilon(x,\omega\cdot\omega',E, E') \domega\dE \domega'\dE'\right)^{1/2}
\\
&\hspace{3.0cm} \left(  \int_{\bV}\int_{\bV}  \left\vert f(x,\omega',E') \right\vert^2 \kappa^\varepsilon(x,\omega\cdot\omega',E, E') \domega\dE \domega'\dE'\right)^{1/2} \\
& = \left(  \int_{\bV}  \left\vert f(x,\omega,E) \right\vert^2 \bar \kappa^\varepsilon(x,\omega,E) \domega\dE\right)^{1/2}
\left(  \int_{\bV}  \left\vert f(x,\omega,E) \right\vert^2 \tilde \kappa^\varepsilon(x,\omega,E) \domega\dE\right)^{1/2}.
\end{align*}
We then have the following, thanks to Young's inequality and our assumption \eqref{eq:hypOP}:
\begin{small}
\begin{align*}
&\left( \cQ^\varepsilon f, f \right)_{\rL^2(\Omega\times\bV)} \\
&\hspace{1.0cm} \geq \int_{\Omega} \, {\rm d}x \int_{\bV} \left\vert f(x,\omega,E) \right\vert^2  \sigma^\varepsilon(x,\omega, E) \domega\dE 
\\
&\hspace{2.0cm} - \int_{\Omega} \left(  \int_{\bV}  \left\vert f(x,\omega,E) \right\vert^2 \bar \kappa^\varepsilon(x,\omega,E) \domega\dE\right)^{1/2}
\left(  \int_{\bV}  \left\vert f(x,\omega,E) \right\vert^2 \tilde \kappa^\varepsilon(x,\omega,E) \domega\dE\right)^{1/2}\, {\rm d}x
\\[0.2 cm]
&\hspace{1.0cm}\geq \alpha \Vert f \Vert^2_{\rL^2(\Omega\times\bV)}.
\end{align*}
\end{small}
\end{proof}
Using Proposition \ref{prop:positivityQ} in the energy arguments, one can obtain the apriori bounds in the following lemma.
\begin{lem}
If $(\sigma^\varepsilon, \kappa^\varepsilon)$ satisfy assumptions \eqref{eq:multiscaleOP} and \eqref{eq:hypOP}, then there exists $C>0$ such that for all $\varepsilon >0$, the solution $\varphi^\varepsilon$ to \eqref{eq:delta-evol-neutron-flux} satisfies
\begin{align}
\left\Vert \varphi^\varepsilon \right\Vert_{\rL^\infty((0,T);\rL^2(\Omega\times\bV))} \le C
\quad\textrm{ and }\quad
\left\Vert \varphi^\varepsilon \right\Vert_{\rL^2((0,T)\times\Omega\times\bV)} \le C.
\end{align}
\end{lem}
%

\subsection{Homogenization result}
\label{sec:hom-neutron}

Following a similar pathway as in our initial ODE model (see Section \ref{sec:ode}), we derive the following result for the scaled Equation \eqref{eq:delta-evol-neutron-flux}, which is proved, without loss of generality, by taking the mass of the neutron $m=2$. Since we focus on the rapid oscillations in energy, in what follows we assume no space dependence in the optical parameters.

\begin{thm}
\label{thm:homog-boltzmann}
Let $\varphi^\varepsilon =\varphi^\varepsilon (t,x,\omega,E)$ be the solution of the equation
\begin{equation}
\left\{
\begin{array}{ll}
\displaystyle\partial_t \varphi^\varepsilon + \sqrt{E}\, \omega \cdot \nabla_x \varphi^\varepsilon  + \sigma^\varepsilon\left(\omega, E \right) \varphi^\varepsilon
- \int_{E_{\rm min}}^{E_{\rm max}} \int_{\vert \omega'\vert=1} \kappa^\varepsilon\left(\omega\cdot\omega', E,E'\right) \varphi^\varepsilon(\omega',E'){\rm d}\omega'\, {\rm d}E' = 0 \\[10pt] 
\varphi^\varepsilon (0,x,\omega,E)=\varphi^\varepsilon _{\rm in}\left(x, \omega,E \right),\\[10pt] 
\varphi^\eps(t,x,\omega,E) = 0 
\quad \forall t>0 \quad \mbox{ and for }(x,\omega)\in \Gamma_-:= \left\{ (x,\omega)\in \partial\Omega\times\bS^{d-1}\ :\ {\bf n}(x)\cdot \omega < 0\right\}.
\end{array}
\right.
\label{eq:boltzmann-thm}
\end{equation}
where the coefficients and the data are of the form
\begin{align*}
\sigma^\varepsilon\left(\omega, E \right)  :=  \sqrt{E}\, \sigma \left( \omega, E, \frac{E}{\eps} \right)\quad
& \mbox{ with }\quad
\sigma\left(\omega,E, y \right) \in \rL^\infty(\bV;\mathrm{C}_{\per}(0,1))
\\
\varphi^\varepsilon_{\rm in} (x,\omega,E) := \varphi _{\rm in}\left(x, \omega,E,  \frac{E}{\eps} \right)\quad
& \mbox{ with }\quad
\varphi _{\rm in}\left(x, \omega,E , y \right)\in \rL^2(\Omega\times \bV;\mathrm{C}_{\per}(0,1))
\\
\kappa^\varepsilon(\omega\cdot\omega',E,E') : = \sqrt{E}\, \kappa_1(\omega\cdot\omega',E) \kappa_2\( \omega\cdot\omega',E', \frac{E'}{\varepsilon} \)
& \mbox{ with }\, 
 \kappa_1(\eta,E) \in \rL^\infty\left([-1,1]\times [E_{\min}, E_{\max}]\right) 
\\
~
 & \mbox{ and }\, 
 \kappa_2\(\eta,E', y' \) \in \rL^\infty\left([-1,1]\times [E_{\min}, E_{\max}];\mathrm{C}_{\per}(0,1)\right).
\end{align*}
Then, 
$$
\varphi^\eps \weak \varphi_{\rm hom} \quad \mbox{ weakly in }\rL^2((0,T)\times\Omega\times \bV)
$$
and $\varphi_{\rm hom}$
satisfies the following partial integro-differential equation
$$
\partial_t \varphi_{\rm hom} + \sqrt{E}\, \omega\cdot\nabla_x \varphi_{\rm hom}+
\sqrt{E}\, \langle{\sigma}\rangle \varphi_{\rm hom} -\int_{E_{\rm min}}^{E_{\rm max}} \int_{\mathbb{S}^{d-1}} \sqrt{E}\, \kappa_1(\omega\cdot\omega',E)
\int_0^1 \kappa_2(\omega\cdot\omega',E',y') \varphi_{\rm hom}{\rm d} y' {\rm d}\omega' {\rm d}E' 
 = 
$$
$$
\int_{E_{\rm min}}^{E_{\rm max}} \int_{\mathbb{S}^{d-1}} \sqrt{E}\, \kappa_1(\omega\cdot\omega',E)
\int_0^1 \kappa_2(\omega\cdot\omega',E',y') 
\left[
e^{-t\sqrt{E'}\mathcal{L}_\sigma} \mathcal{L}_1 \varphi_{\rm in}- \int_0^t e^{-(t-s)\sqrt{E'}\mathcal{L}_\sigma} \sqrt{E'}\mathcal{L}_1 
\sigma(\omega',E',y')\varphi_{\rm hom}\, {\rm d}s \right]
 {\rm d} y' {\rm d}\omega' {\rm d}E'
$$
$$ 
-\sqrt{E}\int_0^1
\sigma(\omega,E,y)\
\left[
e^{-t\sqrt{E}\mathcal{L}_\sigma} \mathcal{L}_1 \varphi_{\rm in}- \int_0^t e^{-(t-s)\sqrt{E}\mathcal{L}_\sigma} \sqrt{E}\mathcal{L}_1 
\sigma(\omega,E,y) \varphi_{\rm hom}\, {\rm d}s \right]
\, {\rm d}y,
$$
with initial condition
$$
\varphi_{\rm hom}(0,x,\omega,E) = \langle \varphi_{\rm in}(x,\omega,E,\cdot) \rangle
$$
and zero absorption condition at the in-flux phase-space boundary.
\end{thm}

\begin{proof}
Following the same strategy as in Theorem \ref{thm:exm:ode}, thanks to the method of characteristics, we can write the explicit solution to \eqref{eq:boltzmann-thm} as
\begin{equation*}
\begin{aligned}
\varphi^\varepsilon& (t,x,\omega,E) = \varphi^\eps_{\rm in}\left(x-\sqrt{E}\omega t, \omega,E\right) e^{-t\sigma^\varepsilon\left(\omega,E\right)}
\\
& + \int_0^t \int_{\bV} \kappa^\varepsilon \left(\omega\cdot\omega',E, E'\right) 
e^{-(t-s)\sigma^\varepsilon\left(\omega,E\right)} \varphi^\varepsilon\left(s,x-(t-s)\sqrt{E}\omega,\omega',E'\right)\, {\rm d}\omega' {\rm d}E' {\rm d}s.
\end{aligned}
\end{equation*}
Making the change of variables $x-\sqrt{E}\omega t=r$, we get
\begin{equation*}
\begin{aligned}
\varphi^\varepsilon& (t,r+\sqrt{E}\omega t,\omega,E) = \varphi^\eps_{\rm in}\left(r, \omega,E\right) e^{-t\sigma^\varepsilon\left(\omega,E\right)} 
\\
& \hspace{2.0cm}+ \int_0^t \int_{\bV} \kappa^\varepsilon \left(\omega\cdot\omega',E, E'\right) e^{-(t-s)\sigma^\varepsilon\left(\omega,E\right)} \varphi^\varepsilon\left(s,r+s\sqrt{E}\omega,\omega',E'\right)\, {\rm d}\omega'\, {\rm d}E'\, {\rm d}s.
\end{aligned}
\end{equation*}
Defining 
\[
\varphi^\varepsilon(t,r+\sqrt{E}\omega t,\omega,E) =: \psi^\varepsilon(t,r,\omega, E)
\]
we have the following expression for the unknown $\psi^\varepsilon$:
\begin{equation}\label{eq:neut:explicit-psi-delta}
\begin{aligned}
\psi^\varepsilon& (t,r,\omega,E) = \varphi^\eps_{\rm in}\left(r, \omega,E\right) e^{-t\sigma^\varepsilon\left(\omega,E\right)} 
\\
& \hspace{2.0cm}+ \int_0^t \int_{\bV}\kappa^\varepsilon \left(\omega\cdot\omega',E, E'\right) e^{-(t-s)\sigma^\varepsilon\left(\omega,E\right)} \psi^\varepsilon\left(s,r,\omega',E'\right)\, {\rm d}\omega'\, {\rm d}E'\, {\rm d}s.
\end{aligned}
\end{equation}
We can pass to the limit in the expression \eqref{eq:neut:explicit-psi-delta} as $\varepsilon\to0$ under the assumptions on the scattering kernel which have been listed in the statement of the theorem.


As the term of interest in the impending asymptotic analysis is the integral term in \eqref{eq:neut:explicit-psi-delta}, we shall detail the asymptotic procedure for that term alone. Multiplying the said integral term by a test function $g\left(E,\frac{E}{\eps}\right)$ and integrating over the $E$ variable and exploiting the structure of the scattering kernel $\kappa^\eps$, we have
\begin{equation}\label{eq:integralTerm}
\begin{aligned}
&\int_{E_{\rm min}}^{E_{\rm max}} \int_0^t \int_{E_{\rm min}}^{E_{\rm max}} \int_{\mathbb{S}^{d-1}}\kappa^\varepsilon \left(\omega\cdot\omega',E, E'\right)  e^{-(t-s)\sigma^\varepsilon\left(\omega,E\right)} \psi^\varepsilon\left(s,r,\omega',E'\right)  g\( E,\frac{E}{\eps} \)  \, {\rm d}\omega'\, {\rm d}E'\, {\rm d}s\, {\rm d}E \\
& \qquad \qquad \qquad = \int_{E_{\rm min}}^{E_{\rm max}} \int_0^t  \int_{\mathbb{S}^{d-1}} \sqrt{E}\, \kappa_1(\omega\cdot\omega',E) e^{-(t-s)\sigma^\varepsilon\left(\omega,E\right)} w^\varepsilon(s,r,\omega,\omega') g\( E,\frac{E}{\eps} \)\, {\rm d}\omega'\, {\rm d}s\, {\rm d}E, 
\end{aligned}
\end{equation}
where we have used the notation
$$
w^\varepsilon(s,r,\omega,\omega') := \int_{E_{\rm min}}^{E_{\rm max}} \kappa_2\(\omega\cdot\omega',E', \frac{E'}{\eps} \)  \psi^\varepsilon\left(s,r,\omega',E'\right) {\rm d}E'.
$$
For every $(s,r,\omega,\omega')\in (0,T)\times \Omega \times \mathbb{S}^{d-1}\times \mathbb{S}^{d-1}$,
$$
w^\varepsilon(s,r,\omega,\omega') \to w^{0} (s,r,\omega,\omega') =\int_{E_{\rm min}}^{E_{\rm max}} \int_0^1 \kappa_2(\omega\cdot\omega',E',y') \psi^0(s,r,\omega',E',y')\,  {\rm d}y'\, {\rm d}E'
$$
pointwise as $\varepsilon\to0$,
where $\psi^0$ is the two-scale limit of $\psi^\varepsilon$. Therefore we have two-scale convergence of \eqref{eq:integralTerm}, namely
\begin{align*}
&\int_{E_{\rm min}}^{E_{\rm max}} \int_0^t  \int_{\mathbb{S}^{d-1}} \sqrt{E}\, \kappa_1(\omega\cdot\omega',E) e^{-(t-s)\sigma^\varepsilon\left(\omega,E\right)} w^\varepsilon(s,y,\omega,\omega') g\( E,\frac{E}{\eps} \)\, {\rm d}\omega'\, {\rm d}s\, {\rm d}E
\\
&\qquad \qquad \weak
\int_{E_{\rm min}}^{E_{\rm max}} \int_0^t  \int_{\mathbb{S}^{d-1}} \int_0^1
\sqrt{E}\, \kappa_1(\omega\cdot\omega',E) e^{-(t-s)\sqrt{E}\sigma\left(\omega,E,y\right)} w^0(s,r,\omega,\omega') g(E,y)
\, {\rm d}y \, {\rm d}\omega'\, {\rm d}s\, {\rm d}E.
\end{align*}
As a result of this calculation, passing to the limit in \eqref{eq:neut:explicit-psi-delta} in the sense of two-scale yields
\begin{align*}
\psi^0& (t,r,\omega,E,y) = \textcolor{black}{\varphi_{\rm in}\left(r, \omega,E,y\right)} e^{-t\sqrt{E}\sigma(\omega,E,y)} 
\\
&\qquad \quad + \int_0^t \int_{E_{\rm min}}^{E_{\rm max}} \int_{\mathbb{S}^{d-1}} \int_0^1
\sqrt{E}\, \kappa_1(\omega\cdot\omega',E) e^{-(t-s)\sqrt{E}\sigma(\omega,E,y)} \kappa_2\(\omega\cdot\omega', E', y' \) \psi^0(s,r,\omega',E',y')
\, {\rm d}y'\, {\rm d}\omega'\, {\rm d}E'\, {\rm d}s
\end{align*}
which implies that $\psi^0$ solves the two-scale evolution
\begin{equation*}
\left\{
\begin{aligned}
\partial_t \psi^0 (t,r,\omega,E,y) & + \sqrt{E}\, \sigma(\omega,E,y) \psi^0 (t,r,\omega,E,y)
\\
& = 
\int_{E_{\rm min}}^{E_{\rm max}} \int_{\mathbb{S}^{d-1}} 
\sqrt{E}\, \kappa_1(\omega\cdot\omega', E)
\int_0^1
 \kappa_2\(\omega\cdot\omega', E', y' \) \psi^0(s,r,\omega',E',y')
\, {\rm d}y'\, {\rm d}\omega'\, {\rm d}E',
\\[0.2 cm]
\psi^0(0,r,\omega,E,y) & = \textcolor{black}{\varphi_{\rm in}\left(r, \omega,E,y\right)}.
\end{aligned}
\right.
\end{equation*}
Proceeding in the same way as in the proof of Theorem \ref{thm:exm:ode}, we decompose the two-scale limit $\psi^0$ into a homogeneous part and a remainder, which is of zero mean over the periodic cell, i.e.,
\begin{align*}
\psi^0(t,r,\omega,E,y) = \psi_{\rm hom}(t,r,\omega,E) + \rho(t,r,\omega,E,y) 
\end{align*}
where
\[
\psi_{\rm hom}(t,r,\omega,E) \coloneqq \langle \psi^0 \rangle (t,r,\omega,E)
\qquad \qquad
\mbox{ and }
\quad\qquad 
\langle \rho \rangle = 0.
\]
We obtain the following coupled system for $\psi_{\rm hom}$ and $\rho$:
\begin{equation}\label{eq:boltz-two-scale}
\left\{
\begin{aligned}
\partial_t \psi_{\rm hom} + \sqrt{E}\, \langle{\sigma}\rangle \psi_{\rm hom} - S( \psi_{\rm hom})
& = S(\rho)
-\sqrt{E}\langle{\sigma\rho}\rangle
\\
\partial_t \rho +\sqrt{E}\left( \sigma \rho - \langle \sigma \rho \rangle \right) &= \sqrt{E}\left(\langle \sigma \rangle -\sigma \right)\psi_{\rm hom}
\\
\psi_{\rm hom}(0,r,\omega,E) &= \langle \varphi_{\rm in}(r,\omega,E,\cdot) \rangle
\\
\rho(0,r,\omega,E,y) &= \mathcal{L}_1 \varphi_{\rm in}(r,\omega,E,y) ,
\end{aligned}\right.
\end{equation}
where
$$
S(\xi):= \int_{E_{\rm min}}^{E_{\rm max}} \int_{\mathbb{S}^{d-1}} \sqrt{E}\, \kappa_1(\omega\cdot\omega',E)
\int_0^1 \kappa_2(\omega\cdot\omega',E',y') \xi(s,r,\omega',E',y') {\rm d} y' {\rm d}\omega' {\rm d}E' 
$$
for any $\xi\in \rL^2((0,T)\times\Omega\times\bV; \mathrm{C}_{\per}(0,1))$. The coupled system written above has the same structure as system \eqref{eq:exm:coupled-system}. Hence, following the same lines as in Section \ref{sec:ode}, we can deduce
$$
\rho(t,r,\omega,E,y)=e^{-t\sqrt{E}\mathcal{L}_\sigma} \mathcal{L}_1 \varphi_{\rm in}(r,\omega,E,y)- \int_0^t e^{-(t-s)\sqrt{E}\mathcal{L}_\sigma} \sqrt{E}\mathcal{L}_1 \sigma(\omega,E,y) \psi_{\rm hom}(s,r,\omega,E) \, {\rm d}s
$$
and, finally, substituting this expression in the evolution equation for $\psi_{\rm hom}$, we obtain
\begin{equation*}
\left\{
\begin{aligned}
& \partial_t \psi_{\rm hom} + \sqrt{E}\, \langle{\sigma}\rangle \psi_{\rm hom} -S( \psi_{\rm hom}) = S(\rho)
-\sqrt{E}\, \langle{\sigma\rho}\rangle \\
& \psi_{\rm hom}(0,r,\omega,E) = \langle \varphi_{\rm in}(r,\omega,E,\cdot) \rangle,
\end{aligned}\right.
\end{equation*}
that is
$$
\partial_t \psi_{\rm hom} + \sqrt{E}\, \langle{\sigma}\rangle \psi_{\rm hom} -\int_{E_{\rm min}}^{E_{\rm max}} \int_{\mathbb{S}^{d-1}} \sqrt{E}\, \kappa_1(\omega\cdot\omega',E)
\int_0^1 \kappa_2(\omega\cdot\omega',E',y') \psi_{\rm hom}{\rm d} y' {\rm d}\omega' {\rm d}E' 
 = 
$$
$$
\int_{E_{\rm min}}^{E_{\rm max}} \int_{\mathbb{S}^{d-1}} \sqrt{E}\, \kappa_1(\omega\cdot\omega',E)
\int_0^1 \kappa_2(\omega\cdot\omega',E',y') 
\left[
e^{-t\sqrt{E'}\mathcal{L}_\sigma} \mathcal{L}_1 \varphi_{\rm in}- \int_0^t e^{-(t-s)\sqrt{E'}\mathcal{L}_\sigma} \sqrt{E'}\mathcal{L}_1 
\sigma(\omega',E',y') \psi_{\rm hom}\, {\rm d}s \right]
 {\rm d} y' {\rm d}\omega' {\rm d}E'
$$
$$ 
-\sqrt{E}
\int_0^1
\sigma(\omega,E,y)\
\left[
e^{-t\sqrt{E}\mathcal{L}_\sigma} \mathcal{L}_1 \varphi_{\rm in}- \int_0^t e^{-(t-s)\sqrt{E}\mathcal{L}_\sigma} \sqrt{E}\mathcal{L}_1 
\sigma(\omega,E,y) \psi_{\rm hom}\, {\rm d}s \right]
\, {\rm d}y,
$$
with initial condition
$$
\psi_{\rm hom}(0,r,\omega,E) = \langle \varphi_{\rm in}(r,\omega,E,\cdot) \rangle.
$$
By going back to the original variables, we finally obtain the thesis of the theorem.
\end{proof}
\begin{rem}
Some comments are in order with regards to the assumption on the optical parameters $\sigma^\varepsilon$ and $\kappa^\varepsilon$ made in Theorem \ref{thm:homog-boltzmann}.
\begin{itemize}
\item The assumption of separability 
\[
\kappa^\varepsilon(\omega\cdot\omega',E,E') : = \sqrt{E}\, \kappa_1(\omega\cdot\omega',E) \kappa_2\( \omega\cdot\omega',E', \frac{E'}{\varepsilon} \)
\]
has simplified our computations in the proof. It also lead to a relatively simpler homogenized model. 
\item In the above separable structure, we can further allow the factor $\kappa_1$ to oscillate in the $E$-variable. The proof of Theorem \ref{thm:homog-boltzmann} can be reworked in this case, at the price of arriving at a more complex two-scale system than \eqref{eq:boltz-two-scale}. The memory structure remains the same but with additional terms.
\item It is apparent from the proof of Theorem \ref{thm:homog-boltzmann} that the energy oscillations in $\sigma$, not those in the scattering kernel, resulted in the memory effects in the homogenized limit.
\end{itemize}
\end{rem}

\section{Numerical tests}
\label{sec:numerics}
We present some results on admittedly simple cases, aiming primarily at studying the convergence rate towards the homogenized solution as the period of the oscillating terms tends to zero.

\subsection{Collision term $\kappa$ inside the integral}
We consider the problem
\begin{equation}\label{eq:ex-numerics-2-eps}
\left\{
\begin{aligned}
\partial_t \varphi^\eps(t,E) + \sigma\left(\frac{E}{\eps}\right) \varphi^\eps(t,E) & = \int_{E_{\min}}^{E_{\max}} \kappa\left(\frac{E'}{\eps}\right) \varphi^\eps(t,E')\, {\rm d}E'
\\[0.2 cm]
\varphi^\eps(0,E) & = \varphi_{\rm in}(E)
\end{aligned}
\right.
\end{equation}
for $(t,E)\in [0,T]\times (E_{\min}, E_{\max})$. This equation is a simple instance of the linear Boltzmann equation \eqref{eq:boltzmann-thm} with neither space nor angular dependence. 
In the following, we set $(E_{\min}, E_{\max})=(0,1)$ and the final time to $T=10$. By Theorem \ref{thm:homog-boltzmann}, we know that $\varphi^\eps$ weakly converges in $\mathrm L^2((0,T)\times (E_{\min}, E_{\max}))$ to $\varphi_{\rm hom}$ under some assumptions on the optical parameters. We study here the convergence rate for different functions $\sigma$, $\kappa$ and initial value $\varphi_{\rm in}$. For this, we introduce the family of orthogonal Legendre polynomials in $\mathrm L^2(E_{\min}, E_{\max})$ which we denote by $\{\ell_k\}_{k\geq 0}$ and define the modes
\begin{equation}
m_k^\eps(t) \coloneqq \left( \varphi^\eps(t,\cdot), \ell_k(\cdot) \right)_{\mathrm L^2([E_{\min}, E_{\max}])},\quad k\geq 0,
\end{equation}
of the solution for $t\in[0,T]$. Likewise,
\begin{equation}
m^{\hom}_k(t) := \left(\varphi_{\hom} (t,\cdot), \ell_k(\cdot) \right)_{\mathrm L^2([E_{\min}, E_{\max}])},\quad k\geq 0,
\end{equation}
are the modes of the homogenized solution $\varphi_{\hom}$. We know that $\varphi^\eps\weak \varphi_{\hom}$ weakly in $\mathrm L^2([E_{\min}, E_{\max}])$ if
\begin{equation}
\label{eq:testConv}
e^\eps_k\coloneqq \max_{t\in[0,T]} | m_k^\eps(t) - m^{\hom}_k(t) | \to 0, \qquad \mbox{ as }\eps\to 0\quad\forall k\geq0.
\end{equation}
In the following, we use this sufficient condition to give numerical evidence of weak convergence. Note that, in practice, we cannot inspect all the values
$k\geq0$ so we present results for the first eight modes.

\medskip
\noindent
\textbf{First example:} We take
\begin{equation*}
\sigma(y)=2+\frac 1 2 \sin(2\pi y),\quad \kappa(y')=1+\frac 1 2 \sin(2\pi y'),\quad \varphi_{\rm in}(y) = 1+\sin(2\pi y).
\end{equation*}
We compute the numerical solution $\varphi^\eps$ with an explicit Runge Kutta method of order four with time step ${\rm d}t=T/50$. We take a uniform mesh of size $h$ for the discretization of the oscillatory variable. We refine $h$ depending on the value of $\eps$ to be able to resolve the oscillations. Our choice is to take $h=\eps/100$. This discretization might be an overkill and one could probably achieve a good accuracy with less degrees of freedom.

We next give numerical evidence that $\varphi^\eps$ weakly converges to $\varphi_{\hom}$ in $\mathrm L^2((0,T)\times (E_{\min},E_{\max}))$ through the plot of the 
the error $e_k^\eps$ in figure \ref{fig:conv-rate-sin-sin-sin-2}. The plot shows that $e_k^\eps$ tends to 0 as $\eps\to0$ for $k=0,\dots,7$ at a rate slightly lower than $\eps$. Next, we illustrate some kind of ``strong'' two-scale convergence by showing convergence in the norm:
\begin{equation}\label{eq:numerical-supp-strong}
\left\vert \left\Vert \varphi^\eps \right\Vert_{\mathrm L^2((0,T)\times (E_{\min},E_{\max}))} - \left\Vert \varphi^0 \right\Vert_{\mathrm L^2((0,T)\times (E_{\min},E_{\max})\times(0,1))} \right\vert \to 0 \qquad \mbox{ as }\eps\to0.
\end{equation}
Figure \ref{fig:conv-rate-sin-sin-sin-2} shows that this is indeed the case. We refer the readers to \cite[Theorem 1.8 on p.1488]{Allaire_1992} for more details on strong two-scale convergence.

\begin{figure}[h!]
\centering
\begin{subfigure}[b]{0.45\textwidth}
\includegraphics[width=\textwidth]{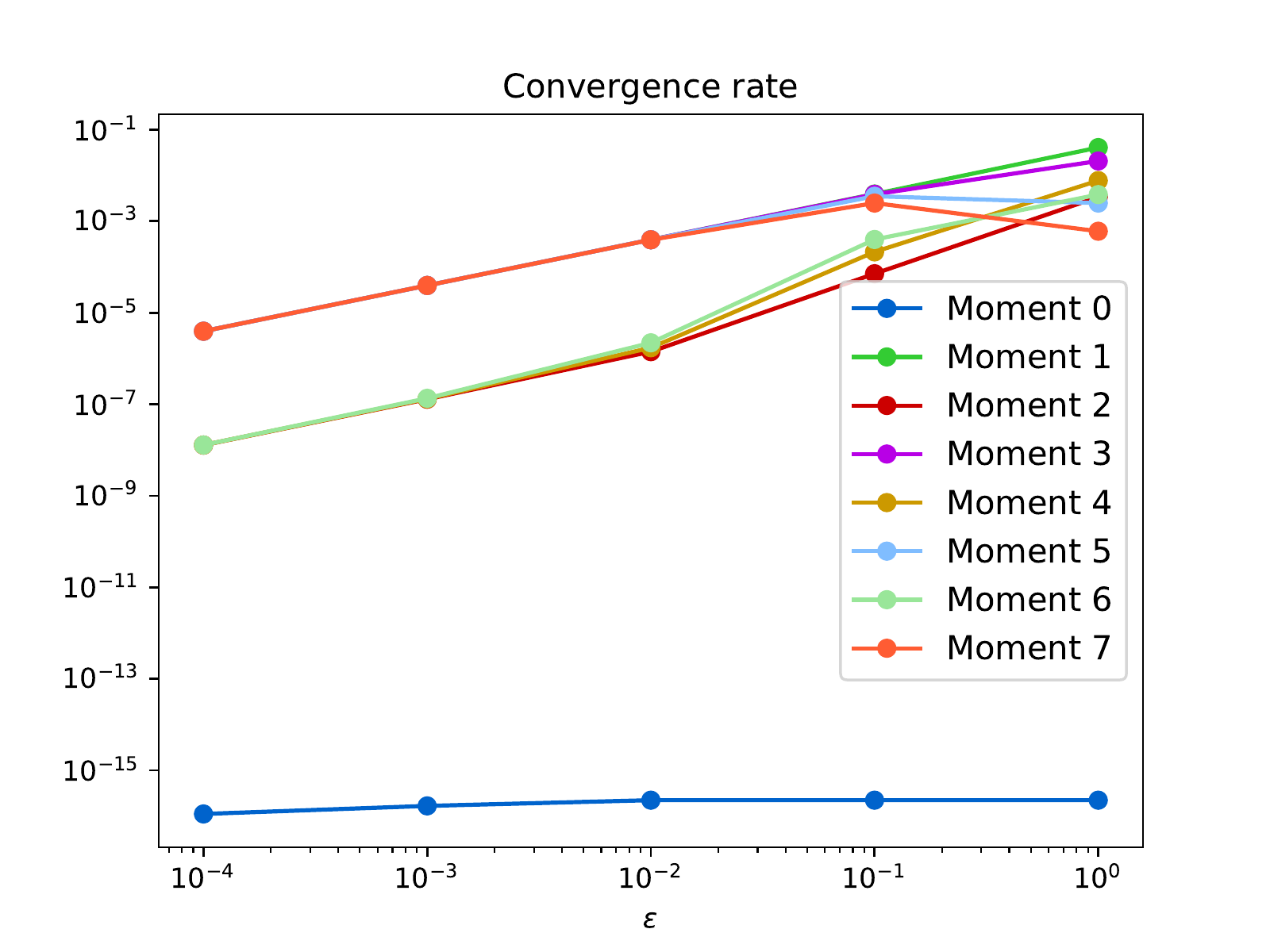}
\caption{Convergence rate in $\eps$ of the error $e_k^\eps$}
\end{subfigure}
~
\begin{subfigure}[b]{0.45\textwidth}
\includegraphics[width=\textwidth]{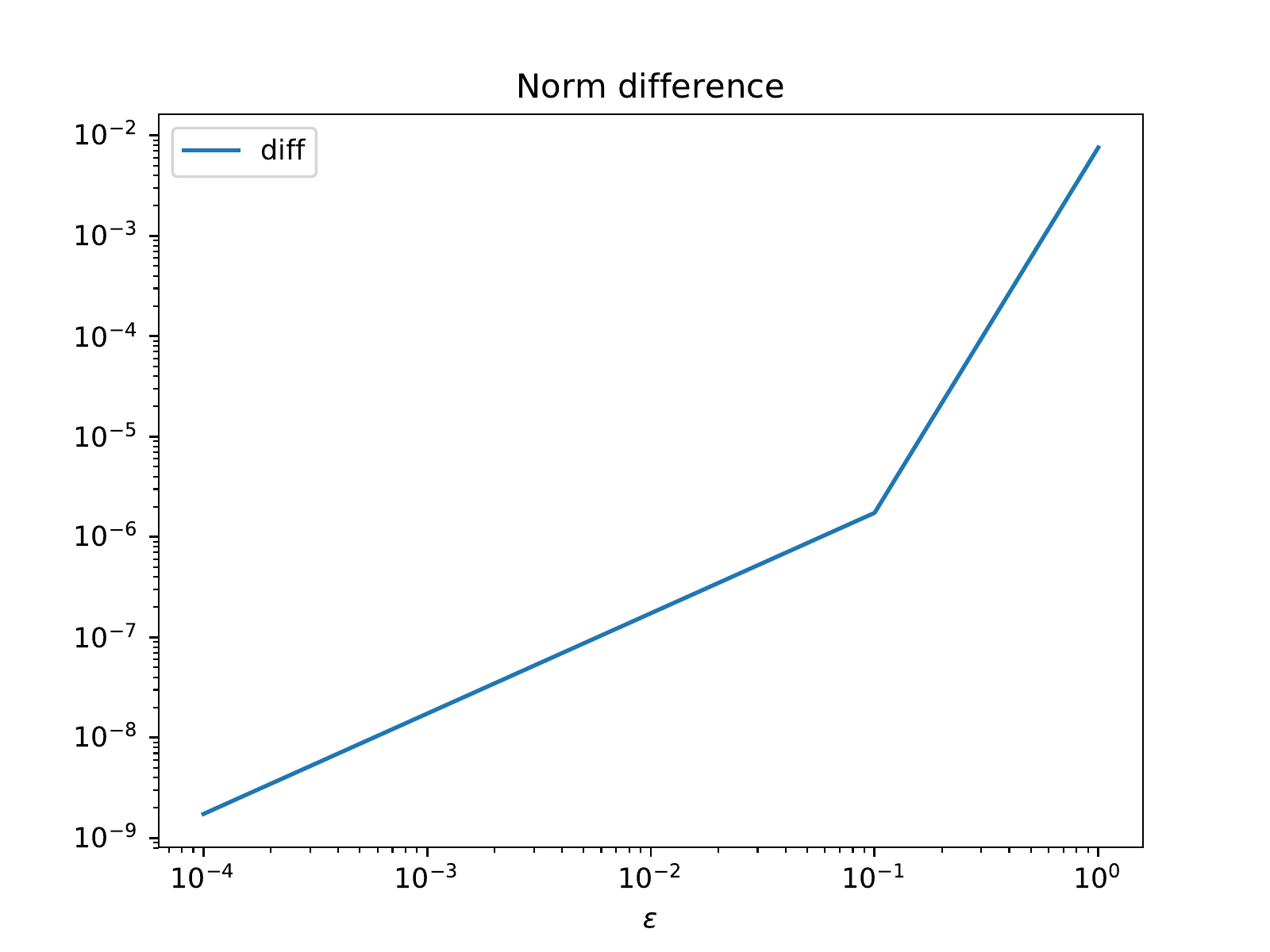}
\caption{Norm difference $\vert \Vert \varphi^\eps \Vert_{\mathrm L^2} - \Vert \varphi^0 \Vert_{\mathrm L^2} \vert$}
\end{subfigure}
\caption{Function $\kappa$ inside the integral. First example.}
\label{fig:conv-rate-sin-sin-sin-2}
\end{figure}

\medskip
\noindent
\textbf{Second example:} Taking the first example as a starting point, we study the impact on the convergence rate of the initial value and consider
\begin{equation*}
\sigma(y)=2+\frac 1 2 \sin(2\pi y),\quad k(y')=1+\frac 1 2 \sin(2\pi y'),\quad \varphi_{\rm in}(y) = 1+\mathds{1}_{\{y\leq \frac12\}}(y).
\end{equation*}
Here and in the following $\mathds{1}_{A}$ denotes the indicator function of a set $A\subset [0,1]$, i.e., $\mathds{1}_{A}(y)=1$ if $y\in A$ and $0$ otherwise. As a result, in our example $\varphi_{\rm in}(y)$ has a jump at $y=\frac12$. Figure \ref{fig:conv-rate-haar-sin-sin-2} shows $e_k^\eps$ as a function of $\eps$. We observe that the discontinuity in the initial condition does not affect the overall convergence rate. Here again, we observe that \eqref{eq:numerical-supp-strong} holds.
\begin{figure}[h!]
\centering
\begin{subfigure}[b]{0.45\textwidth}
\includegraphics[width=\textwidth]{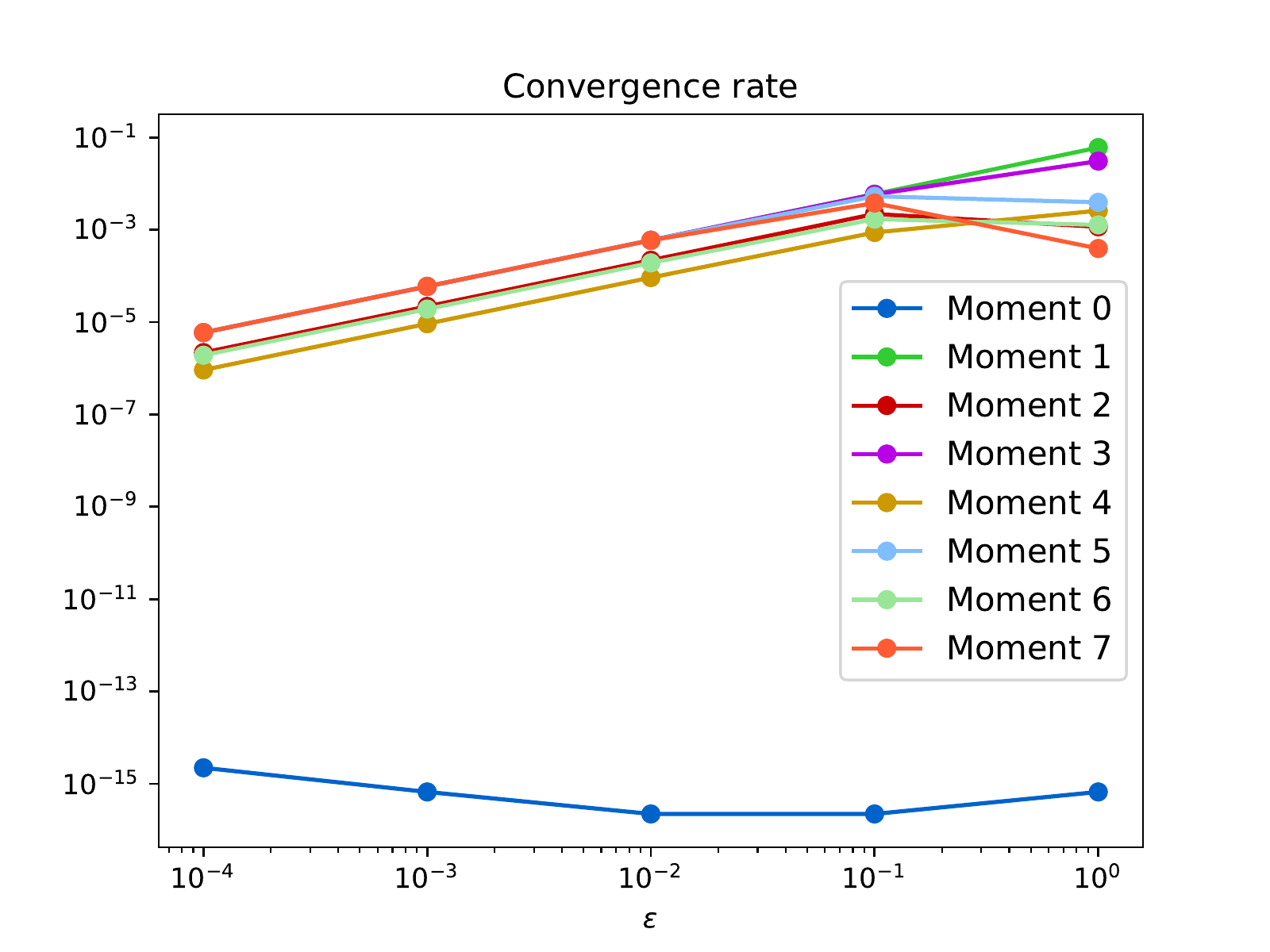}
\caption{Convergence rate in $\eps$ of the error $e_k^\eps$}
\end{subfigure}
~
\begin{subfigure}[b]{0.45\textwidth}
\includegraphics[width=\textwidth]{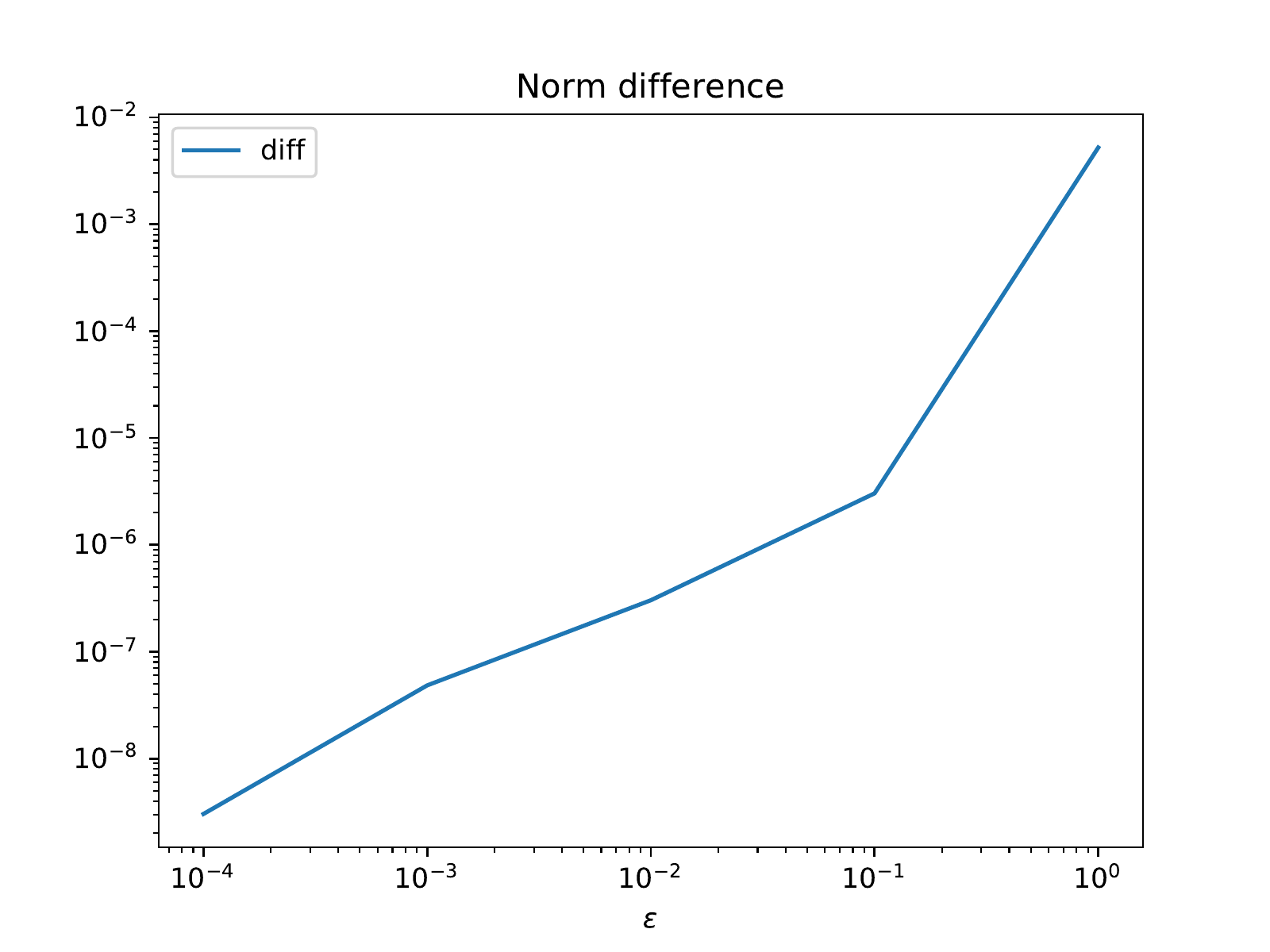}
\caption{Norm difference $\vert \Vert \varphi^\eps \Vert_{\mathrm L^2} - \Vert \varphi^0 \Vert_{\mathrm L^2} \vert$}
\end{subfigure}
\caption{Function $\kappa$ inside the integral. Second example.}
\label{fig:conv-rate-haar-sin-sin-2}
\end{figure}

\medskip
\noindent
\textbf{Third example:} Taking once again the first example as a starting point, we now investigate the impact on the convergence rate that oscillatory and discontinuous functions $\sigma$ and $k$ can have. We consider
\begin{equation*}
\sigma(y)=2+\frac 1 2 \mathds{1}_{\{\sin(2\pi y) \geq0  \}}(y),\quad \kappa(y')=1+\frac 1 2 \mathds{1}_{\{\sin(2\pi y') \geq0  \}}(y'),\quad \varphi_{\rm in}(y) = 1+\sin(2\pi y)
\end{equation*}
and figure \ref{fig:conv-rate-sin-haar-haar-2} shows the convergence of $e_k^\eps$ and of the norm difference. We see that weak convergence occurs at a similar rate as before, indicating that the regularity of the oscillatory functions might have a mild impact on the convergence.

\begin{figure}[h!]
\centering
\begin{subfigure}[b]{0.45\textwidth}
\includegraphics[width=\textwidth]{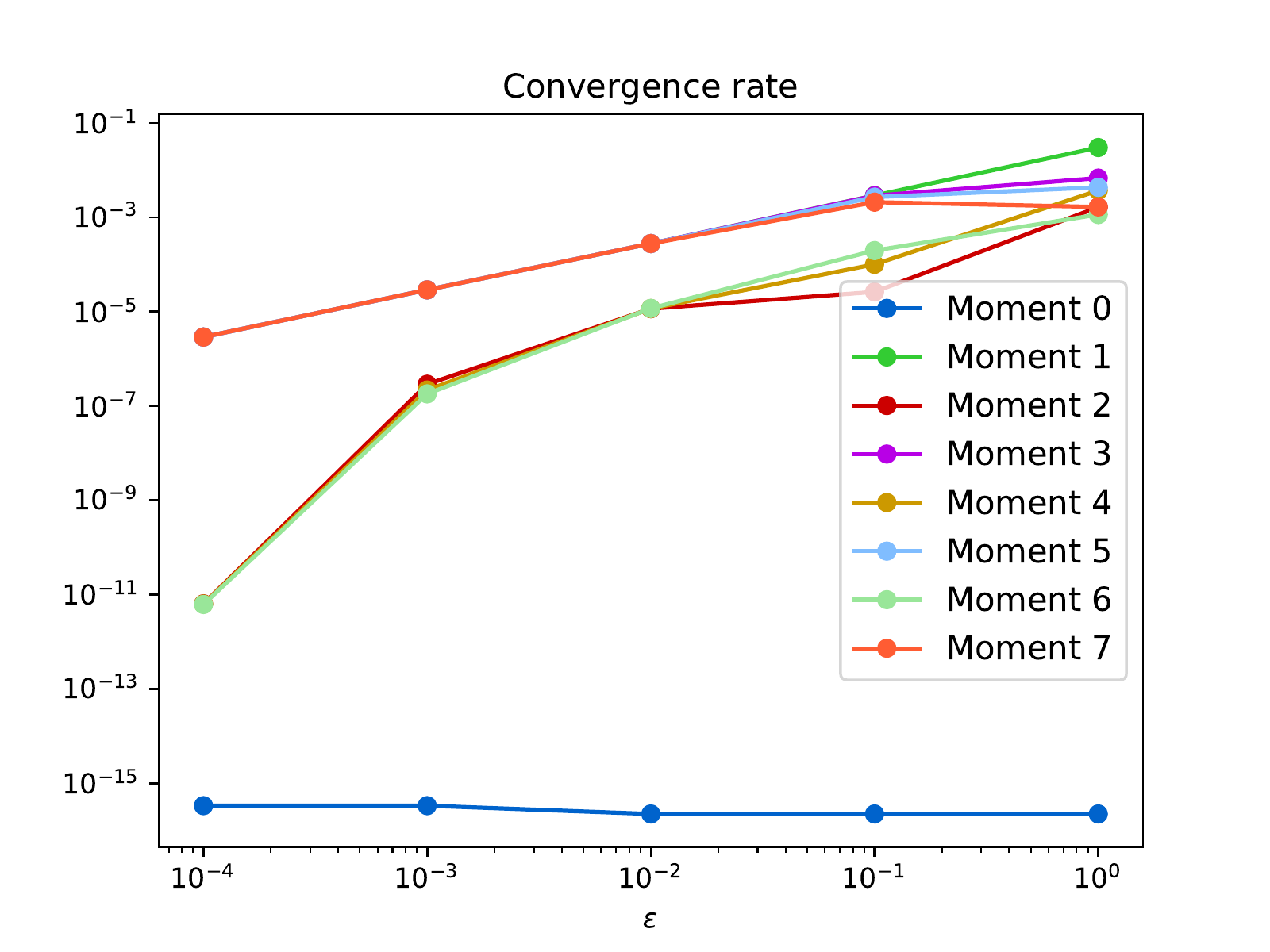}
\caption{Convergence rate in $\eps$ of the error $e_k^\eps$}
\end{subfigure}
~
\begin{subfigure}[b]{0.45\textwidth}
\includegraphics[width=\textwidth]{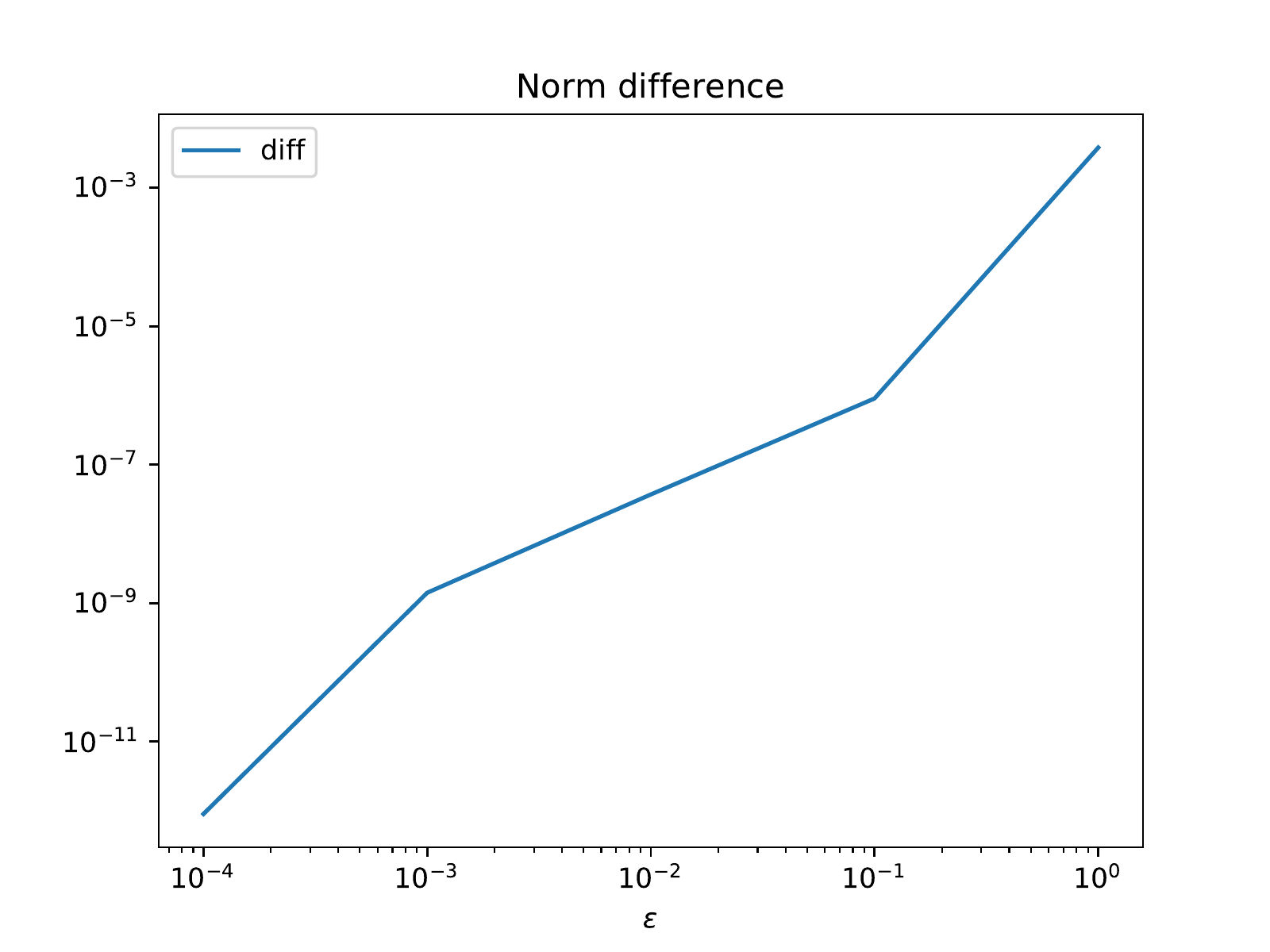}
\caption{Norm difference $\vert \Vert \varphi^\eps \Vert_{\mathrm L^2} - \Vert \varphi^0 \Vert_{\mathrm L^2} \vert$}
\end{subfigure}
\caption{Function $\kappa$ inside the integral. Third example.}
\label{fig:conv-rate-sin-haar-haar-2}
\end{figure}

\subsection{Collision term $\kappa$ outside the integral}
It is interesting to compare the converge rate of problem \eqref{eq:ex-numerics-2-eps} with
\begin{equation}\label{eq:ex-numerics-eps}
\left\{
\begin{aligned}
\partial_t \varphi^\eps(t,E) + \sigma\left(\frac{E}{\eps}\right) \varphi^\eps(t,E) & = \kappa\left(\frac{E}{\eps}\right) \int_{E_{\min}}^{E_{\max}} \varphi^\eps(t,E')\, {\rm d}E'
\\[0.2 cm]
\varphi^\eps(0,E) & = \varphi_{\rm in}(E)
\end{aligned}
\right.
\end{equation}
where the collision kernel $\kappa$ is outside of the integral. Figure \ref{fig:conv-rate-k-outside} gives numerical evidence of weak and strong convergence of \eqref{eq:ex-numerics-eps} for the same three choices of $\sigma$, $\kappa$ and $\varphi_{\rm in}$ that were studied in the previous section. Note that, due to the fact that $\kappa$ is now outside of the integral, the convergence rate is higher and occurs at a rate of about $\eps^2$. Note also that in this case the regularity of the initial condition and the oscillatory coefficients significantly affect the convergence rate. For all cases, we also observe strong convergence (the plots are very similar to the ones of the previous section -- so we omit them).

\begin{figure}[h!]
\centering
\begin{subfigure}[b]{0.45\textwidth}
\includegraphics[width=\textwidth]{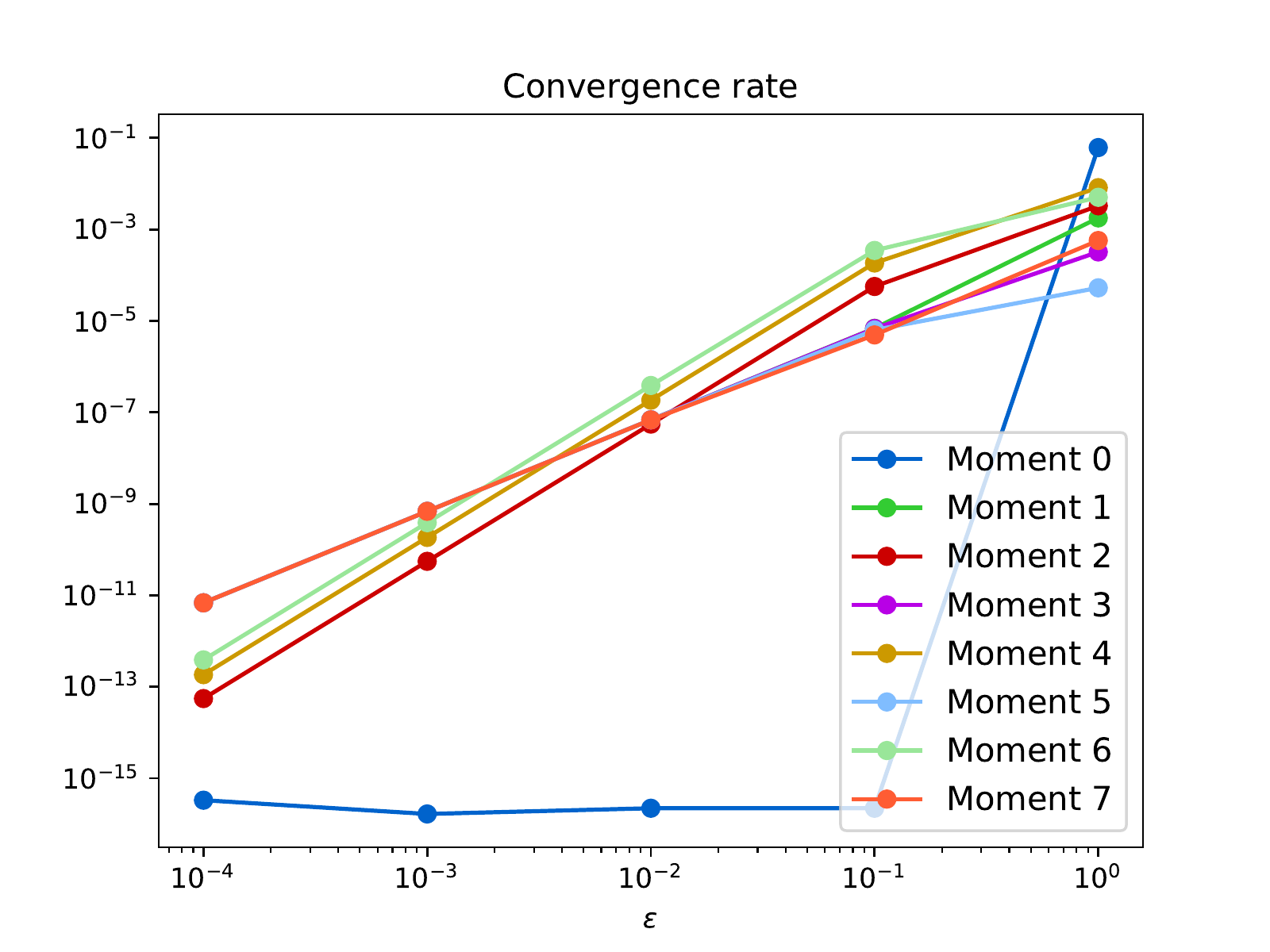}
\caption{First example}
\end{subfigure}
~
\begin{subfigure}[b]{0.45\textwidth}
\includegraphics[width=\textwidth]{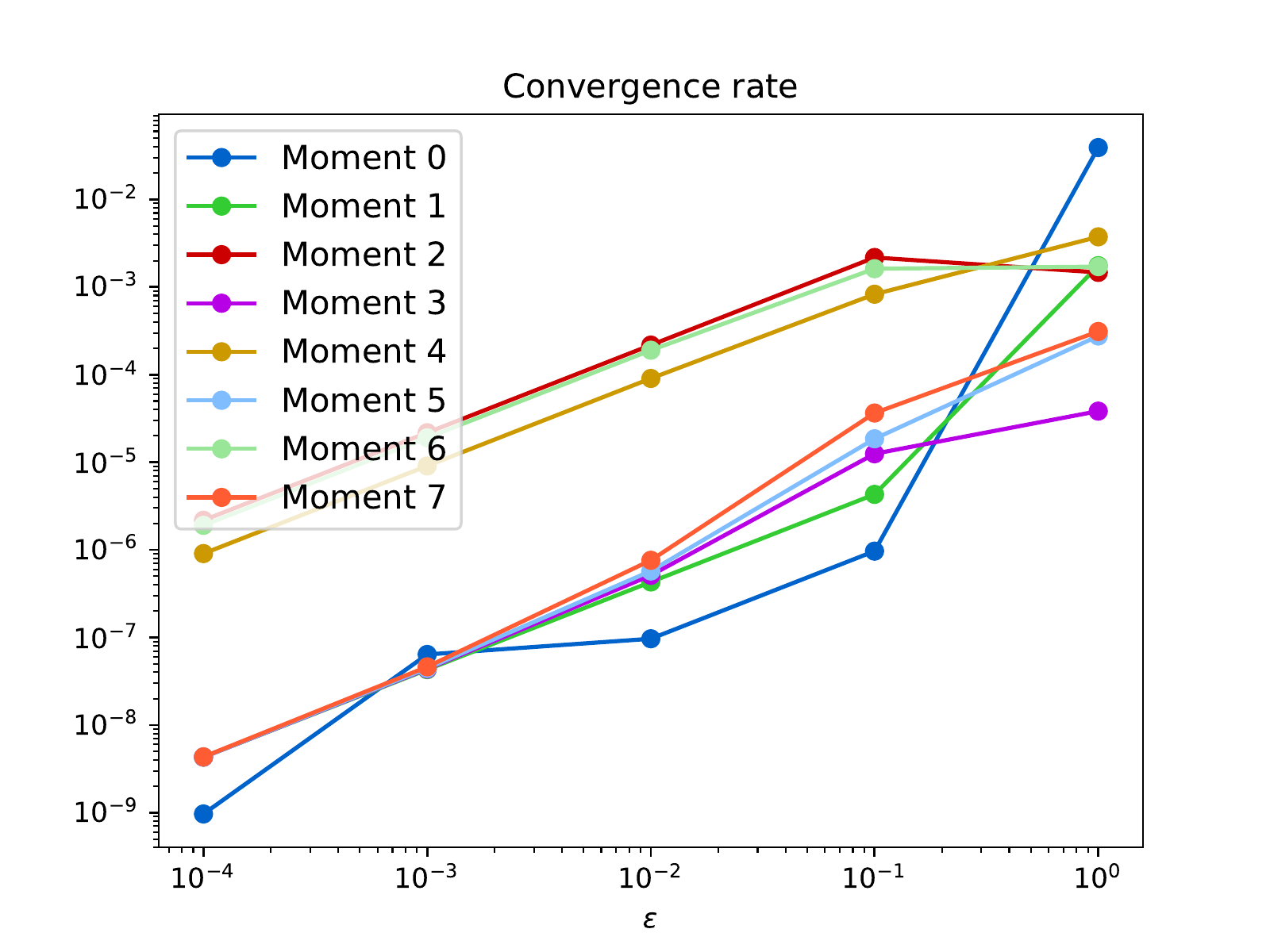}
\caption{Second example}
\end{subfigure}
\\
\begin{subfigure}[b]{0.45\textwidth}
\includegraphics[width=\textwidth]{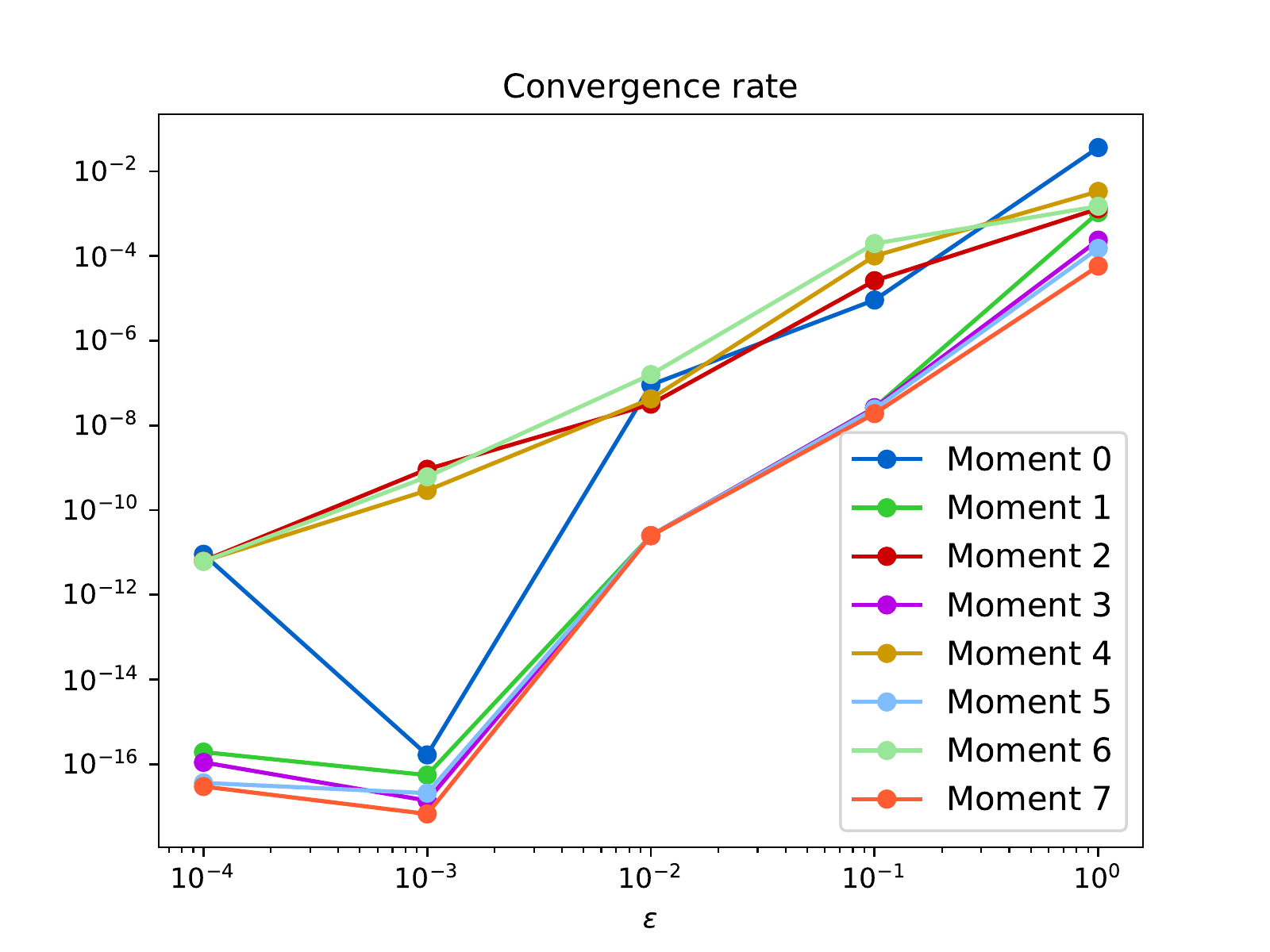}
\caption{Third example}
\end{subfigure}
\caption{Function $\kappa$ outside the integral: Convergence rate in $\eps$ of the error $e_k^\eps$.}
\label{fig:conv-rate-k-outside}
\end{figure}

\appendix
\section{Harmonic oscillator}
\label{sec:complex}
Taking inspiration from Tartar's work on memory effects \cite{Tartar1989}, we analyse the asymptotic behaviour of a differential equation for the unknown $u^\eps:[0,\infty)\times\Omega\mapsto \mathbb{C}$
\begin{equation}\label{eq:model-imag}
\left\{
\begin{aligned}
\partial_t u^\eps(t,x) + i\, b^\eps(x) u^\eps(t,x) & = 0 \quad \qquad \mbox{ for }t>0, \, x\in\Omega,
\\[0.1 cm]
u^\eps(0,x) & = u_{\rm in}(x) \quad \mbox{ for }x\in\Omega.
\end{aligned}\right.
\end{equation}
The coefficient family $\{b^\eps(x)\}$ is assumed to be a family of bounded real-valued function. Taking $u^\eps = w^\eps + i v^\eps$, we obtain a system of equations for the real and imaginary parts:
\begin{equation}\label{eq:system-w-v}
\left\{
\begin{aligned}
\partial_t w^\eps(t,x) - b^\eps(x) v^\eps(t,x) & = 0 \qquad \quad \quad \qquad \mbox{ for }t>0, \, x\in\Omega,
\\[0.1 cm]
\partial_t v^\eps(t,x) + b^\eps(x) w^\eps(t,x) & = 0 \qquad \quad  \quad \qquad \mbox{ for }t>0, \, x\in\Omega,
\\[0.1 cm]
\left( w^\eps, v^\eps \right)(0,x) & = \left( w_{\rm in}, v_{\rm in} \right)(x) \quad \mbox{ for }x\in\Omega,
\end{aligned}\right.
\end{equation}
where we have used that $u_{\rm in}(x) = w_{\rm in}(x) + i \, v_{\rm in}(x)\mbox{ for }x\in\Omega$. Define the vector function $\mathtt{U}^\eps:[0,\infty)\times\Omega\mapsto \R^2$
\[
\mathtt{U}^\eps(t,x) := 
\left( \begin{matrix}
w^\eps
\\
v^\eps
\end{matrix}
\right)(t,x).
\]
Consider a skew-symmetric matrix
\[
\mathtt{A} := 
\left(
\begin{matrix}
0 & 1
\\[0.2 cm]
-1 & 0
\end{matrix}
\right).
\]
Then the system of equations \eqref{eq:system-w-v} can be rewritten as
\begin{equation}\label{eq:ode-system}
\left\{
\begin{aligned}
\partial_ t \mathtt{U}^\eps(t,x) & = b^\eps(x)\, \mathtt{A}\, \mathtt{U}^\eps(t,x) \, \, \quad \qquad \mbox{ for }t>0, \, x\in\Omega,
\\[0.2 cm]
\mathtt{U}^\eps(0,x) & = \mathtt{U}_{\rm in}(x) := \left( \begin{matrix}
w_{\rm in}
\\
v_{\rm in}
\end{matrix}
\right)(x) \quad \mbox{ for }x\in\Omega.
\end{aligned}\right.
\end{equation}
Note that the system \eqref{eq:ode-system} is nothing but the harmonic oscillator equation with heterogeneous coefficients. Note further that memory effects have been shown to appear in the homogenization limit for the second order ordinary differential equations (see the work of J-S.~Jiang, K.-H .~Kuo and C-K.~Lin  \cite{JKL2005}). 
\begin{thm}\label{thm:homogen-harmonic}
Let $\mathtt{U}^\eps(t,x)$ be the solution to \eqref{eq:ode-system}. Suppose that coefficient family $b^\eps(x)$ is a bounded family of real-valued functions which satisfy
\[
-\infty < b_{\rm min} \le b^\eps(x) \le b_{\rm max} < \infty
\]
uniformly in $\eps$ and in $x\in\Omega$. Then we have
\[
\mathtt{U}^\eps \weak \mathtt{U}^0 \qquad \qquad \mbox{ weakly }\ast \mbox{ in }\mathrm L^\infty
\]
with the limit $\mathtt{U}^0$ satisfying the following integro-differential system:
\begin{equation}\label{eq:limit-equation-U0-appendix}
\left\{
\begin{aligned}
\partial_t \mathtt{U}^0(t,x) + b^\ast(x) \mathtt{A} \mathtt{U}^0(t,x) & = \int_0^t \mathbb{K}(x,t-s) \mathtt{U}^0(s,x)\, {\rm d}s \qquad \mbox{ for }t>0, \, x\in\Omega,
\\[0.2 cm]
\mathtt{U}^0(0,x) & = \mathtt{U}_{\rm in}(x) \qquad \qquad \qquad \qquad \qquad  \mbox{ for }x\in\Omega.
\end{aligned}\right.
\end{equation}
The memory kernel in the limit equation is characterised in terms of its Laplace transform in the time variable. More precisely, we have
\begin{align}\label{eq:Laplace-memory-kernel}
\widehat{\mathbb{K}}(p,x) = p\, {\rm Id} + b^\ast(x)\, \mathtt{A} - \mathtt{B}(p,x)
\qquad
\mbox{ for }p>0, \, x\in\Omega
\end{align}
where the matrix $\mathtt{B}$ is 
\begin{align}\label{eq:defn-matrix-B}
\mathtt{B}(p,x) := \left( \int_{b_{\rm min}}^{b_{\rm max}} \frac{1}{p^2+\lambda^2} \left( \begin{matrix} p & \lambda \\[0.2 cm] -\lambda & p\end{matrix}\right) \, {\rm d}\nu_x(\lambda) \right)^{-1}.
\end{align}
where $\{\nu_x\}_{x\in\Omega}$ is the family of Young measures associated with the bounded family $b^\eps(x)$.
\end{thm}
Our proof of the above theorem, similar to the analysis of Luc Tartar in \cite{Tartar1989}, goes via the Laplace transform. Furthermore, as we are making no structural assumptions (such as periodicity) on the coefficient family $b^\eps(x)$, we need to use the notion of Young measures. Hence, before we give the proof of Theorem \ref{thm:homogen-harmonic}, we recall a cornerstone result of the theory of compensated compactness which gives the existence of certain probability measures called the \textsc{young mesaures} \cite{Tartar_1979}.
\begin{thm}\label{thm:young-measures}
Suppose $[\alpha, \beta]\subset \R$ be a bounded interval and let $\Omega\subset\R^d$ be an open set. Consider the family (indexed by $\eps$) of measurable functions $f^\eps(x):\Omega\mapsto\R$ such that 
\[
f^\eps(x) \in [\alpha, \beta]\qquad \mbox{ for }x\in\Omega,
\]
uniformly in $\eps$. Then, there exists a subsequence (still indexed $f^\eps(x)$) and there exists a family of probability measures $\{\nu_x\}_{x\in\Omega}$ with ${\rm supp}\, \nu_x\subset[\alpha,\beta]$ such that if $\mathcal{G}\in\mathrm C(\R)$, then
\begin{align}
\mathcal{G}(f^\eps)(x) \weak \int_\alpha^\beta \mathcal{G}(\lambda) \, {\rm d}\nu_x(\lambda) \qquad \mbox{ weakly }\ast\mbox{ in }\rL^\infty(\Omega),
\end{align}
i.e., for any arbitrary $\varphi\in\rL^1(\Omega)$, we have
\[
\lim_{\eps\to0} \int_\Omega \mathcal{G}(f^\eps)(x) \varphi(x)\, {\rm d}x = \int_\Omega \int_\alpha^\beta \mathcal{G}(\lambda) \varphi(x) \, {\rm d}\nu_x(\lambda) \, {\rm d}x.
\]
\end{thm}

\begin{proof}[Proof of Theorem \ref{thm:homogen-harmonic}]
Taking the Laplace transform of the evolution in \eqref{eq:ode-system} results in
\begin{align*}
p\, \widehat{\mathtt{U}^\eps}(p,x) - \mathtt{U}_{\rm in}(x) = b^\eps(x)\, \mathtt{A} \, \widehat{\mathtt{U}^\eps}(p,x),
\end{align*}
which yields
\begin{align}\label{eq:family-LU-eps}
\widehat{\mathtt{U}^\eps}(p,x) = \frac{1}{p^2 + \left\vert b^\eps(x) \right\vert^2}
\left(
\begin{matrix}
p & b^\eps(x)
\\[0.2 cm]
-b^\eps(x) & p
\end{matrix}
\right) \mathtt{U}_{\rm in}(x).
\end{align}

For each fixed $p>0$, the above expression defines a family of functions (functions of the $x$ variable) indexed by $\eps$. By fixing a $p>0$, we let $\eps$ go to zero and arrive at a limit point for the family $\left\{\widehat{\mathtt{U}^\eps}(p,x)\right\}$. Now, in Theorem \ref{thm:young-measures}, let us take the continuous function $\mathcal{G}(s):\R\to\R$ from one of the below choices
\[
\frac{p}{p^2 + s^2}; \qquad
\frac{s}{p^2 + s^2}.
\]
Then, there exists a subsequence (still indexed $b^\eps(x)$) and there exists a family of probability measures $\{\nu_x\}_{x\in\Omega}$ with ${\rm supp}\, \nu_x\subset[b_{\rm min},b_{\rm max}]$ such that we have the following $\rL^\infty(\Omega)$ weak $\ast$ limits:
\begin{align*}
\frac{p}{p^2 + \left\vert b^\eps(x) \right\vert^2} & \weak \int_{b_{\rm min}}^{b_{\rm max}} \frac{p}{p^2 + \lambda^2}\, {\rm d}\nu_x(\lambda),
\\[0.3 cm]
\frac{b^\eps(x)}{p^2 + \left\vert b^\eps(x) \right\vert^2} & \weak \int_{b_{\rm min}}^{b_{\rm max}} \frac{\lambda}{p^2 + \lambda^2}\, {\rm d}\nu_x(\lambda).
\end{align*}
We are now equipped to pass to the limit in the family \eqref{eq:family-LU-eps}:
\begin{align*}
\widehat{\mathtt{U}^\eps}(p,x) \weak & \mathtt{V}^0(p,x) := \left( \int_{b_{\rm min}}^{b_{\rm max}} \frac{1}{p^2+\lambda^2} \left( \begin{matrix} p & \lambda \\[0.2 cm] -\lambda & p\end{matrix}\right) \, {\rm d}\nu_x(\lambda) \right) \mathtt{U}_{\rm in}(x)
\end{align*}
weakly $\ast$ in $\rL^\infty(\Omega)$. We have hence proved
\begin{align}\label{eq:intermediate-uo-vo}
\mathtt{B}(p,x) \mathtt{V}^0(p,x) = \mathtt{U}_{\rm in}(x)
\end{align}
where the matrix $\mathtt{B}$ is defined in \eqref{eq:defn-matrix-B}. We rewrite \eqref{eq:intermediate-uo-vo} as
\begin{align}\label{eq:limit-V0}
p\, \mathtt{V}^0(p,x) + b^\ast(x) \mathtt{A} \mathtt{V}^0(p,x) - \mathtt{U}_{\rm in}(x) = \Big( p\, {\rm Id} + b^\ast(x)\, {\rm Id} - \mathtt{B}(p,x) \Big) \mathtt{V}^0(p,x)
\end{align}
where
\[
b^\ast(x) := \int_{b_{\rm min}}^{b_{\rm max}} \lambda \, {\rm d}\nu_x(\lambda)
\]
is the $\rL^\infty(\Omega)$ weak $\ast$ limit of the family $b^\eps(x)$. Taking the inverse Laplace transform of \eqref{eq:limit-V0}, we find that the inverse Laplace transform of the weak limit $\mathtt{V}^0$ (which coincides with $\mathtt{U}^0$) satisfies \eqref{eq:limit-equation-U0-appendix}.
\end{proof}

\begin{rem}
Note that the above derivation of the memory effects induced by homogenization for \eqref{eq:model-imag} has given an expression for the memory kernel in terms of its Laplace transform. It is of interest, however, to give the expression of the kernel $\mathbb{K}$ in terms of the real $t$ variable. This needs taking the inverse Laplace of the expression \eqref{eq:Laplace-memory-kernel}.
\end{rem}

\section*{Acknowledgments}
This work was supported by the ANR project \textit{Kimega} (ANR-14-ACHN-0030-01) and by the Italian Ministry of Education, University and Research (\textit{Dipartimenti di Eccellenza} program 2018-2022, Dipartimento di Matematica 'F. Casorati', Universit\`a degli Studi di Pavia). We also received funding from a PEPS grant from the INSMI (Institut National des Sciences Mathématiques et de leurs Interactions). H.~Hutridurga and O.~Mula started to work on this paper while enjoying the kind hospitality of the Hausdorff Institute for Mathematics in Bonn during the trimester program on multiscale problems in 2017. H. Hutridurga also acknowledges the support of the EPSRC programme grant ``Mathematical fundamentals of Metamaterials for multiscale Physics and Mechanics'' (EP/L024926/1). The authors thank Vishal Vasan (ICTS) for posing the question on harmonic oscillator addressed here in the Appendix \ref{sec:complex}.

\bibliographystyle{bibstyle}
\bibliography{references}

\end{document}